\documentclass[a4paper,12pt,reqno]{amsart}
\usepackage{verbatim}
\usepackage{enumitem}
\usepackage{hyperref}
\usepackage{amssymb}
\usepackage[all]{xy}

\newcommand{\ZZ}{\mathbb{Z}}

\newcommand{\x}{\times}

\newcommand{\la}{\langle}
\newcommand{\ra}{\rangle}

\newcommand{\Q}{\mathbb{Q}}
\newcommand{\G}{\Gamma}

\newcommand{\im}{\mathrm{im}\,}         
\newcommand{\Z}{\mathbb{Z}}

\newcommand{\N}{\mathbb{N}}

\newcommand{\id}{\mathrm{Id}}   

\newcommand{\cA}{\mathcal{A}}
\newcommand{\cB}{\mathcal{B}}
\newcommand{\cM}{\mathcal{M}}

\DeclareMathOperator{\coker}{coker}

\newtheorem{theorem}{Theorem}[section]
\newtheorem{lemma}[theorem]{Lemma}
\newtheorem{corollary}[theorem]{Corollary}
\newtheorem{proposition}[theorem]{Proposition}

\newtheorem{example}[theorem]{Example}
\newtheorem{conjecture}[theorem]{Conjecture}
\newtheorem{definition}[theorem]{Definition}

\newtheorem{thmx}{Theorem}

\theoremstyle{remark}
\newtheorem{remark}[theorem]{Remark}

\begin{document}

\title{On strongly inflexible manifolds}

\author[C. Costoya]{Cristina Costoya}
\address{CITIC, Departamento de Computaci\'on,
Universidade da Coru{\~n}a, 15071-A Coru{\~n}a, Spain.}
\email{cristina.costoya@udc.es}

\author[V. Mu\~{n}oz]{Vicente Mu\~{n}oz}
\address{Departamento de \'Algebra, Geometr\'{\i}a y Topolog\'{\i}a, Universidad de M\'alaga,
Campus de Teatinos, s/n, 29071 M\'alaga, Spain}
\email{vicente.munoz@uma.es}

\author[A. Viruel]{Antonio Viruel}
\address{Departamento de \'Algebra, Geometr\'{\i}a y Topolog\'{\i}a, Universidad de M\'alaga,
Campus de Teatinos, s/n, 29071 M\'alaga, Spain}
\email{viruel@uma.es}

\thanks{The first author was partially supported by MINECO (Spain) grant PID2020-115155GB-I00.
The second author was partially supported MINECO (Spain) grant PID2020-118452GB-I00.
The third author was partially supported by MINECO (Spain) grant PID2020-118753GB-I00. }


\begin{abstract}
 An oriented closed connected $N$-manifold $M$ is inflexible if it does not admit self-maps of unbounded degree. In addition, if all the maps from any other oriented closed connected $N$-manifold have bounded degree, then $M$ is said to be strongly inflexible. The existence of simply-connected inflexible manifolds was established by Arkowitz and Lupton. However, the existence of  simply-connected strongly inflexible manifolds is still an open question. We provide an algorithm relying on Sullivan models that allows us to prove that all, but one, of the known examples of simply-connected inflexible manifolds are not strongly inflexible.

\end{abstract}

\maketitle

\section{Introduction}\label{sec:intro}

Let $\operatorname{Mfd}_N$ be the class of oriented closed connected manifolds of dimension $N$.
Given $M \in \operatorname{Mfd}_N$ we are interested in the set
 $$
 \deg(M)\buildrel\text{\scriptsize def}\over{:=} \{ \deg(f)\, |\, f\colon M\to M \text{ continuous}\}\subset\mathbb{Z}.
 $$
Constant maps and the identity map show that $0,1\in \deg(M)$, and $-1\in \deg(M)$ unless $M$ is not strongly chiral \cite{Mu}. Moreover, $\deg(M)$ is a multiplicative semi-group, hence if there exists any $\ell\in \deg(M)$ such that $|\ell|>1$, then $\deg(M)$ is  unbounded.
We say that $M$ is \emph{inflexible} if $\deg(M)$ is bounded.

The existence of inflexible manifolds is intimately related to the existence of nontrivial functorial seminorms  (see \cite{CL} for more details).
An important example is the functorial $l^1$-seminorm, studied by Gromov, which is nontrivial on the fundamental class of oriented closed connected hyperbolic manifolds. Hence, hyperbolic manifolds are inflexible. However, there exist inflexible manifolds where either the $l^1$-seminorm is trivial or not known \cite{Neo1, Neo2}. 
In particular, the $ l^1$-seminorm is trivial on the nonzero homology classes of simply-connected spaces, thus raising Gromov's question \cite[5.35 Remarks.(b)]{Gro99} on whether finite nontrivial functorial seminorms on singular homology exist for simply-connected spaces.  A step towards understanding this problem is to ask whether simply-connected inflexible manifolds do exist. The first example in literature of a simply-connected inflexible manifold is established by Arkowitz and Lupton \cite[Example 5.1, 5.3 Remarks]{AL2} using Rational homotopy theory techniques. Every example of a simply-connected inflexible manifold constructed at the time of writing has been built using Sullivan models \cite{Am, CosMenVir18,CosMenVir20,CV2, CL}.

Suppose now that $M$ and $M'$ are in $\operatorname{Mfd}_N$  and let
 $$
  \deg(M',M)\buildrel\text{\scriptsize def}\over{:=}\{\deg(f)\, |\, f\colon M'\to M \text{ continuous}\}\subset\mathbb{Z}.
  $$
Map composition results in a multiplicative action of the semigroup $\deg(M)$ on the set $\deg(M',M)$. Therefore $0\in\deg(M',M)$, and if there exists a nonzero $d\in \deg(M',M)$ and $M$ is not
  inflexible, then $\deg(M',M)$ is unbounded. We say that $M$ is \emph{strongly inflexible} if for every $N$-manifold $M'$ the set $\deg(M',M)$ is bounded. The bound might depend on $M'$, for instance, if we take the connected sum of $n$ copies of $M$, $M'=\#_n M$,  there exists a degree $n$
  map $f:M'\to M$. Observe that if $M$ is strongly inflexible, then $M$ is automatically inflexible.

The interesting aspect of a strongly inflexible manifold  $M$,  is that a domination $\operatorname{Mfd}_N$-seminorm associated with $M$ can be defined \cite[Definition 7.1]{CL}
   $$
    v_{M}(M')\buildrel\text{\scriptsize def}\over{:=}\sup\big\{|d|\, |\, d\in\deg(M',M)\big\},
    $$
which  can be extended to a nontrivial finite functorial  seminorm in singular homology of simply-connected spaces. Hence, if strongly inflexible manifolds were to exist, Gromov's open question \cite[5.35 Remarks.(b)]{Gro99}  would be settled in  the positive,  and  mapping degree theorems for simply-connected manifolds could be deduced.  However, no example of simply-connected strongly inflexible manifolds are known.

The condition of being inflexible or strongly inflexible can be both translated to the language of Sullivan algebras for simply-connected (or even
nilpotent) spaces. Let $(\Lambda V,d)$ be a Sullivan algebra such that its cohomology $H(\Lambda V, d) $ is a Poincar{\'e} duality algebra of formal dimension $N$ (\cite[Section 38]{FHT}).
Suppose that $V^1=0$, and let $\nu \in (\Lambda V)^N$ be a representative of the fundamental class. Then $(\Lambda V,d)$
is inflexible if for every dga morphism
 $$
 \varphi\colon(\Lambda V,d) \to (\Lambda V,d)
 $$
we have $\deg(\varphi)=0,\pm 1$, where $H^N([\nu])=\deg(\varphi) [\nu]$.
Moreover,  the Sullivan algebra $(\Lambda V,d)$ is strongly inflexible if for every Sullivan algebra $(\Lambda W,d)$ with Poincar\'e duality cohomology of formal dimension $N$,
the set of mapping degrees
 $$
  (\Lambda V,d) \to (\Lambda W,d)
 $$
is bounded. In this paper, using the notion of (strongly) inflexibility in dgas and its relation with the (non) existence of positive weights (see Definition \ref{def:weight}),  a criteria for a manifold  $M \in \operatorname{Mfd}_N$  to be not strongly inflexible is given:

\begin{thmx}[Theorem \ref{thm:weigth_implies_integral_flexible}]\label{thm:A}
Let $M \in \operatorname{Mfd}_N$, $N\geq 4$,  with minimal model $(\Lambda V,d)$ and let $\eta$ be its cohomological fundamental class. Write the rational cohomological fundamental class as $\eta \otimes_\mathbb Q 1=[\nu]$.
Assume there exists a dga morphism $\psi\colon (\Lambda V,d)\to (\mathcal{A},d)$ where $(\mathcal{A},d)$ is a simply-connected finite type dga with positive weight and $H^{N}(\psi)([\nu]) \ne 0$. Then $M$ is not strongly inflexible.
\end{thmx}

Applying  Theorem \ref{thm:A} we prove that \emph{all but one} of the known examples of simply-connected inflexible dgas and manifolds,
as given in \cite{Am, AL2, CosMenVir18, CosMenVir20, CV2, CL},  are not strongly inflexible (see Remark \ref{rem:Am-fail}
for the example we have not yet been able to prove that is not strongly inflexible):

\begin{thmx}[Theorems \ref{thm:dga_Tables} and \ref{thm:acta_mapsto_positive_weight}, Remark \ref{rem:otherexamples} ]\label{thm:B}
Let $\cM$ be one of the inflexible dgas from  \cite[Example\ 3.8]{Am},  \cite[Examples\ 5.1 and 5.2]{AL2},  \cite[Definition\ 1.1]{CosMenVir18},  \cite[Definition 4.1]{CosMenVir20},  \cite[Definition 2.1]{CV2}, \cite[Examples I.1--I.4]{CL}.   Then, $\cM$ is not strongly inflexible.  Furthermore,  any manifold for which $\cM$ is a Sullivan model  is not strongly inflexible either.
 \end{thmx}

These results lead us to raise the following conjecture, that  would imply that all finite functorial seminorms vanish on simply-connected manifolds:
\begin{conjecture}\label{conjecture}
Simply-connected strongly inflexible manifolds do not exist.
\end{conjecture}

In \cite[Section 5.3]{Neo14} it is asked whether simply-connected closed inflexible $N$-manifolds can be dominated by a product $X_1 \times X_2$ where one of the factors is not an inflexible manifold. If so, then Conjecture \ref{conjecture} would be true.

\subsection*{Acknowledgements} The authors want to thank C.\ Neofytidis for useful conversations and drawing our attention to \cite{SWWZ}. The authors also want to thank the anonymous referees for their constructive comments and suggestions.
\section{Weights} \label{sec:weight}
We work with differential graded algebras (dga) with coefficients over $\Q$.
In this section we discuss the key ingredient in this work, namely, the relationship between weights in a dga and the strong inflexibility property. Our definition of weight is taken from \cite{Body-Douglas-1, Body-Douglas-2}.
\begin{definition}\label{def:weight}
Let $(\cA, d)$
be a connected dga. We say that $(\cA,d)$ has a weight $\omega$ if each degree component $\cA^k$ of $\cA$ decomposes as a direct sum
 $$
 \cA^k=\bigoplus\limits_{n\in\Z}{}_n \cA^k,\qquad k\geq 0,
 $$
such that
\begin{enumerate}[label={\rm (\roman{*})}]
\item $d({}_n \cA^k)\subset {}_n \cA^{k+1}$, and
\item ${}_n \cA^k\wedge {}_m \cA^l\subset {}_{n+m}\cA^{k+l}$.
\end{enumerate}
Every nonzero $x\in {}_n\cA=\bigoplus\limits_{k\geq 0}{}_n\cA^k$ is said to have weight $n$,
and it is denoted by $\omega(x)=n$. An element $x\in \cA$ is said $\omega$-homogeneous if there exists $n \in \Z$ such that $x\in {}_n\cA$.
\end{definition}

We now give some notions of weights.

\begin{definition}\label{def:types-weight}
Let $(\cA, d)$ be a connected dga with a weight $\omega$.
\begin{enumerate} [label={\rm (\roman{*})}]
\item Given $a\in \cA$ we say that $\omega$ detects $a$ if $a$ is $\omega$-homogeneous and $\omega(a)\ne 0$.
\item We say that $\omega$ is nontrivial if it detects some nonzero element $a\in \cA$.
\item We say that $\omega$ is positive if ${}_n\cA= 0$ for every $n< 0$ and ${}_0\cA=\cA^0$.
\end{enumerate}
\end{definition}

\begin{lemma}\label{lem:homogeneos}
Let $(\cA,d)$ be a connected dga with a weight $\omega$. Then,  there is an induced weight in $ H(\cA)$, also denoted by $\omega$, given by $w ([a]) = w(a)$  for every $w$-homogeneous element $a \in \cA$. Hence, given $x\in H(\cA,d)$, there exists a decomposition
$x=\sum\limits_{i=0}^r[a_i]$, where the elements $a_i\in \cA$ are $\omega$-homogeneous cocyles, and $\omega(a_i)=\omega(a_j)$ if and only if $i=j$.
\end{lemma}

\begin{proof}
Notice that we can decompose $\cA=\bigoplus \limits_{n \in \Z}   {}_n\cA$ where each $({}_n\cA, d)$ is a differential graded module. Therefore $H(\cA,d)= \bigoplus \limits_{n \in \Z}    H({}_n\cA, d)$ which induces a weight $\omega$ in $H(\cA,d)$ by defining ${}_nH(\cA,d):=H({}_n\cA, d)$.
\end{proof}

\begin{corollary}\label{cor:topclass_homogenea}
Let $(\cA,d)$ be a connected dga with Poincar\'e duality cohomology of
formal dimension $N$. If  $(\cA,d)$ has a weight $\omega$, then there exists a $w$-homogeneous cocycle   $\nu$ such that $[\nu]$ is the fundamental class of $H^N(\cA).$
\end{corollary}

\begin{proof}
According to Lemma \ref{lem:homogeneos}, we have $[\nu]=\sum\limits_{i=0}^r [\nu_i]$, where the elements $\nu_i\in \cA$
are  $\omega$-homogeneous cocycles, and $\omega(\nu_i)=\omega(\nu_j)$ if and only if $i=j$. Since $H^N(\cA,d)\cong\Q$,
and $H^N(\cA,d)$ has an induced weight, it must be that only one $i_0$ satisfies $[\nu_{i_0}]\neq 0$. Hence $[\nu]=[\nu_{i_0}]$  with
$\nu_{i_0}$ anf $\omega$-homogeneous cocycle.
\end{proof}

\begin{definition}\label{def:flexible}
Let $(\cA,d)$ be a connected dga. A class $x\in H(\cA,d)$ is said flexible if for every
$n\in\Z$ there exists a dga morphism $f_n\colon \cA\to \cA$ such that $H(f_n)(x)=\lambda \, x$ with $\lambda \geq n$.
\end{definition}

The existence of nontrivial weights implies the existence of flexible classes.

\begin{lemma}\label{lem:33}
Let $(\cA,d)$ be a connected dga, and $[a]\in H(\cA,d)$. If $(\cA,d)$
has a weight $\omega$ that detects the element $a\in\cA$, then $[a]$ is a flexible class.
\end{lemma}

\begin{proof}
For $q\in\Q$, let $\varphi_{q}\colon(\cA,d) \rightarrow (\cA,d)$ be the following dga morphism: for $y$ an $\omega$-homogeneous element, $\varphi_q(y) = q^{\omega(y)} y.$  Otherwise,   we decompose $y = \sum\limits_{i=0}^r y_i$ into $\omega$-homogeneous elements and $\varphi_q(y) =\sum\limits_{i=0}^r q^{\omega(y_i)} y_i.$

Now, for every $n\in\Z$, since $\omega$ detects  $a$, $\omega(a)\neq 0$ and therefore there exists $q_n\in\Q$ such that
$q_n^{\omega(a)}\geq n$. By defining $f_n = \varphi_{q_n} \colon\cA\to \cA$ we get that  $H(f_n)([a])= q_n^{\omega(a)}\, [a]$ with $q_n^{\omega(a)}\geq n$. Hence $[a]$ is a flexible class.
\end{proof}

\begin{corollary}\label{cor:weight_implies_flexible}
Let $(\Lambda V ,d)$ be the Sullivan model of  a simply-connected dga $(\cA,d)$ with Poincar\'e duality cohomology of formal dimension $N$.  If $(\cA,d)$ has a weight $\omega$ that detects a representative of the fundamental class, then $(\Lambda V ,d)$  is not inflexible.
\end{corollary}

\begin{proof}
Since $\omega$ detects
$\nu$, then by Lemma \ref{lem:33} the fundamental class $[\nu]$ is a flexible class. Then, for every $n\in\Z$ there exists a dga morphism $f_n\colon \cA\to \cA$
such that $\deg(f_n)=\lambda\geq n$. Finally, according to the Lifting Lemma \cite[Lemma 12.4]{FHT},  $f_n$ lifts to $\widetilde{f}_{n}$,
a self-morphism of $(\Lambda V ,d)$ whose degree is $\deg(\widetilde{f}_{n})=\deg(f_n)=\lambda\geq n$, and therefore  $(\Lambda V ,d)$ is not inflexible.
\end{proof}

\begin{proposition} \label{prop:11}
Let $(\cA,d)$ be a simply-connected, finite type dga. For any nonzero class $[\nu]\in H^N(\cA,d)$, $N\geq 4$, there is a dga $(\bar \cA,d)$ whose cohomology is Poincar\'e duality of formal dimension $N$, and a dga morphism
  $$
  q\colon (\cA,d)\to (\bar \cA,d),
   $$
such that $H^N(q)([\nu]) \neq 0$ is a fundamental class for $H^*(\bar \cA,d)$.
\end{proposition}

\begin{proof}
 We take a simply-connected finite type CW-complex $X$ whose Sullivan model is isomorphic to $(\cA,d)$, \cite[Theorem 10.2 (ii)]{Su}.
 Let $X'$ be the $(N+1)$-skeleton of $X$, which is a finite complex. We have that $H^{\leq N}(X'; \Q) \cong H^{\leq N}(X; \Q)\cong H^{\leq N}(\cA,d)$.  By \cite[Th{\'e}or{\`{e}}mes III.4]{Thom}, the cohomology class $[\nu]\in H^{N}(X', \Q)$ is detected by an oriented closed $N$-manifold $M$, or in other words,  there exists a map $f\colon M\to X'$ such that $H^N(f, \Q)( [\nu] )= a [M]_\Q$,  $a \in \Q$ nontrivial,  where $[M]_\Q$ denotes the rational cohomological fundamental class of $M$.  Moreover, we can assume that $M$ is simply-connected (see \cite[Corollary 3.2]{CL}).

Let $(\bar \cA,d)$ be a model of $M$. The cohomology of $(\bar\cA,d)$ is a Poincar\'e duality algebra of formal dimension $N$. Then $f'\colon M\to X'\subset X$ is represented by a dga morphism $q\colon (\cA,d)\to (\bar \cA,d)$ such that $H^N(q)([\nu])\neq 0$.
\end{proof}

\begin{remark}
The assumption of simply-connectedness in Proposition \ref{prop:11} is not necessary.  The dga  $(\cA,d)$ can be considered nilpotent. Hence, the space $X$ is nilpotent too and by  \cite[Theorem 3.1(2)]{CL}, we can assume that there exists a continuous map $f: M \rightarrow X'$ such that $\pi_1(M) \rightarrow \pi_1(X')$ is an isomorphism. Therefore $M$ is also nilpotent and it admits a model. \end{remark}

\begin{corollary} \label{cor:44}
Let $(\cA,d)$ be a connected dga with Poincar\'e duality cohomology and  fundamental class $[\nu]\in H^N(\cA,d)$, $N\geq 4$. Suppose
that there exists a dga $(\cA',d)$ and a morphism $\varphi\colon (\cA,d)\to (\cA',d)$ such that
$H^N(\varphi)([\nu])\neq 0$. If $(\cA',d)$ has a positive weight $\omega$,
then $(\cA,d)$ is not strongly inflexible.

\end{corollary}

\begin{proof}
By Lemma \ref{lem:homogeneos},  we decompose $H^N(\varphi)([\nu])=
\sum\limits_{i=0}^r [a'_i]$ where $a'_i$ are $\omega$-homogeneous cocycles and  $\omega(a'_i) \neq \omega(a'_j)$ if $i \neq j$. Assume that
$\omega(a_0') = \max \{\omega(a_i')| i=0, \ldots, r\}$. Notice that since $\omega$ is a positive weight, it detects all the $\omega$-homogeneous elements in positive degree.

In the proof of Lemma \ref{lem:33} we have shown how to construct,  for every $n \in \N$, a morphism
 $
  f_n\colon (\cA',d) \to (\cA',d)
  $
verifying that $H^N(f_n)( [a_i']) = q_n^{\omega(a_i')} [a_i']$ for every $i$, where $q_n$ is a natural number satisfying that $q_n^{\omega(a_0')} \geq n$.

Now, associated to $(\cA',d)$,
by Proposition \ref{prop:11} there is a  dga $(\bar\cA,d)$ whose cohomology is a Poincar\'e duality algebra of formal dimension $N$, and a dga morphism  $
  q\colon (\cA',d)\to (\bar \cA,d)
   $
such that $H^N(q)([a_0']) \neq 0$ is a fundamental class, let us call it $[\mu]$, for $H^*(\bar\cA,d)$.

Finally, take $G_n=q\circ f_n\circ \varphi\colon (\cA,d)\to (\bar\cA,d)$, which in cohomology satisfies that:
$$H^N(G_n) ([\nu] )= (q_n^{\omega(a_0')}  + \sum \limits_{1}^r q_n^{\omega(a_i')} \alpha_i )[\mu], \; \alpha_i \in \mathbb Q.$$
Now, observe $P(x)=x^{\omega(a_0')}  + \sum \limits_{1}^r x^{\omega(a_i')} \alpha_i$ is a rational monic polinomial of degree $\omega(a_0')$. As $q_n^{\omega(a_0')} \geq n$, the set $\{\deg (G_n) | \,n \in \mathbb N \}$ (which coincides with $\{P(q_n)|\, q_n^{\omega(a_0')} \geq n\in \mathbb N \}$) is unbounded.
\end{proof}

\begin{proposition} \label{prop:55}
Every $2$-step Sullivan algebra admits a positive weight.
\end{proposition}

\begin{proof}
A $2$-step Sullivan algebra is of the form $(\Lambda (V_1\oplus V_2), d)$ such that
$d(V_1)=0$, and $d(V_2)\subset \Lambda V_1$. Let $\{x_i \mid i \in I\}$ be a degree homogeneous basis of $V_1$ and
$\{y_j \mid j\in J\}$ a degree homogeneous basis of $V_2$. Then a positive weight  $\omega$ is defined by declaring
the elements of the basis $\omega$-homogeneous of weight
$$
 \omega (x_i)= |x_i|,\, \omega(y_j)=|y_j|+1
 $$
and extending this weight to monomials in $(\Lambda (V_1\oplus V_2), d)$ according to the rule in Definition \ref{def:weight}(ii).
\end{proof}

\section{Universal spaces}\label{sec:universal}

In this section we develop  the necessary tools to extend our previous results on non-strong inflexibility from dgas to manifolds.
Recall that given $p$, a prime or zero, a map $f\colon X\to Y$ is said to be a $p$-equivalence
if $f$ induces an isomorphism on $H^*(X;\Z/p)\cong H^*(Y;\Z/p)$. Here $\Z/0\buildrel\text{\scriptsize def}\over{:=}\Q$.

Combining \cite{BMSS, MNT, MT}, a finite CW-complex
 $X$ is said to be universal if for every $p$, 
 for any given $p$-equivalence $k\colon Y\to Z$, and for every map $g\colon X\to Z$, there is a map $h\colon X\to Y$ and there is a $p$-equivalence $f\colon X\to X$ such that the following diagram commutes up to homotopy:
\begin{equation*}
\xymatrix{
X \ar@{.>}[d]^h\ar@{.>}[r]^f & X\ar[d]_g\\
Y\ar[r]^k & Z.
}
\end{equation*}

Universal spaces are characterized by their minimal models (\cite[Theorem A]{BMSS}):

\begin{theorem}\label{thm:universal-weights}
Let $X$ be a simply-connected finite CW-complex, and $(\Lambda V, d)$ be its minimal model. Then $X$ is universal if and only if $(\Lambda V, d)$ admits a positive weight.
\end{theorem}

The cohomology of universal spaces is generated by flexible classes.

\begin{theorem}\label{thm:universal_implies_flexible}
Let $M$ be a manifold in $\operatorname{Mfd}_N$ with cohomological fundamental class $\eta\in H^N(M;\Z)$.
Assume that there exists a map $g\colon X\to M$ such that $X$ is a simply-connected finite universal CW-complex,  $H^N(X;\mathbb{Q})\cong \Q$, and
$H^N(g; \Q) (\eta_\Q) \ne 0$, where $\eta_\Q = \eta \otimes_\Q 1$. Then $M$ is not strongly inflexible.
\end{theorem}

\begin{proof}
Let $(\Lambda V, d)$ be the Sullivan minimal model of $X$, which by Theorem \ref{thm:universal-weights} admits a positive weight $\omega.$  Then, by Lemma  \ref{lem:homogeneos}, there exists a decomposition of  $H^N(g; \Q) (\eta_\Q) \ne 0$  into $\omega$-homogeneous elements:
$$H^N(g; \Q) (\eta_\Q) = [v]\in H^N(X; \mathbb Q) \cong H^N (\Lambda V, d)  \cong \mathbb Q.$$
As we have shown in the proof of Lemma \ref{lem:33}, for any $n \in \mathbb N$,  we can choose an integer $q_n$ such that $ q_n^{\omega(v)} \geq n$.
Then there exists a dga morphism:
$$\begin{array}{rcl}
f_n\colon (\Lambda V, d)&\to&(\Lambda V, d)\\
x&\mapsto & q_n^{\omega(x)} x,
\end{array}$$
for every $\omega$-homogeneous element $x$ of  $(\Lambda V, d)$.

Let $\zeta_X\colon X \rightarrow X_\Q$ be the rationalization of $X$, that is, the localization at $\Q$, and let $\psi_n\colon X_\Q \rightarrow X_\Q$ be the realization of $f_n.$
Using that $X$ is $0$-universal, there exists a $0$-equivalence ${\widetilde{\psi}_n}\colon X \rightarrow X$ and a map $h_n\colon X \rightarrow X$ such that the following diagram homotopy commutes:
\begin{equation}\label{eq:diagrama_universal_l}
\begin{gathered}
\xymatrix{
X\ar@{.>}[rr]^{\widetilde{\psi}_n}  \ar@{.>}[d]_{h_n}&& X \ar[d]^{\zeta_X} \\
X \ar[r]^{\zeta_X}& X_\Q \ar[r]^{\psi_n} & X_\Q.
}
\end{gathered}
\end{equation}
Observe that $ \widetilde{\psi}_n, \, \psi_n$ and $\zeta_X$ are $0$-equivalences, thus $h_n$  is also a $0$-equivalence.

Using that the diagram \eqref{eq:diagrama_universal_l} is homotopy commutative, we get that
\begin{equation}\label{diagramcommutes}
H^N(\widetilde{\psi}_n; \mathbb Q ) ([v]) =H^N(  \psi_n \circ  \zeta_X\circ h_n; \mathbb Q ) ([v]) = H^N(h_n; \mathbb Q)( l_n^{\omega(v)} [v]) = c_n l_n^{w(v)}[v],
\end{equation}
where $c_n $ is a nonzero integer since $h_n$ is a $0$-equivalence.

Let $\nu = H^N(g; \Z)(\eta) \in H^N(X; \mathbb Z)$. By \cite[Th\'eor\`emes III.4]{Thom}, there exists an $N$-manifold $M'$ with cohomological fundamental class $\mu$,
and a map $\theta\colon M'\to X$ such that $H^N(\theta; \Z)(\nu)=k\mu$ for some integer $k\neq 0$.
We are going to show that $\deg(M',M)$ is unbounded, thus $M$ is not strongly inflexible.

Indeed, for every $n \in \mathbb N$, we can define the composition:
$$ \xymatrix{
G_n \colon M' \ar[r]^-{\theta} & X  \ar[r]^{\widetilde{\psi}_n} & X \ar[r]^{g} &M. \\
}
$$
Then,
\begin{align*}
H^N(G_n; \Z) (\eta)
&=H^N (\widetilde{\psi}_n \circ \theta; \Z)(H^N(g; \Z)(\eta)) \\
&=H^N(\theta; \Z)(  H^N({\widetilde{\psi}_n}; \Z )(\nu)) && \text{(we now use equation \eqref{diagramcommutes})}\\
&=H^N( \theta, \Z) (c_nl_n^{\omega(v)} \nu+ t) && \text{(where $t$ is some torsion element)}\\
&=c_n l_n^{\omega(v)}H^N(\theta; \Z)(\nu)   &&\text{($H^N(\theta; \Z)(t) =0$ since $H^N(M'; \mathbb Z) \cong \mathbb Z$)}\\
&=kc_nl_n^{\omega(v)} \mu.
\end{align*}
Therefore $|\deg(G_n)| = |kc_nl_n^{\omega(v)}| \geq n$, as $k$ and $ c_n$ are nonzero integers.
\end{proof}

\begin{remark}
Observe that from Theorem \ref{thm:universal_implies_flexible} we deduce that a universal manifold  $M \in \operatorname{Mfd}_N$ is not inflexible (see also \cite[Corollary 4.1]{Man}). 
\end{remark}

\begin{proposition} \label{prop:pesos_en_modelos}
Let $(\mathcal{A},d)$ be a simply-connected dga with a positive weight $\omega$. Then there exists a minimal model $\rho\colon(\Lambda V,d) {\buildrel \sim\over\longrightarrow} (\mathcal{A},d)$ such that
$(\Lambda V,d)$ has a positive weight $\widetilde{\omega}$ with
$\omega(\rho(v))=\widetilde{\omega}(v)$, for every $\widetilde{\omega}$-homogeneous element $v\in \Lambda V$.
\end{proposition}

\begin{proof}
We construct $\rho$ inductively. Suppose that $\rho: \Lambda V^{<n} \rightarrow \cA$ is constructed in such a way that, on the one hand, $ \Lambda V^{<n}$ admits a positive weight $\widetilde{\omega}$ with $\omega(\rho(v))=\widetilde{\omega}(v)$, for every $\widetilde{\omega}$-homogeneous element $v\in \Lambda V^{<n}$. And, on the other hand,  $H^{<n}(\rho)$ is an isomorphism and $H^{n}(\rho)$ is a monomorphism.

By construction, the morphism $H (\rho):   H (\Lambda V^{<n})  \rightarrow H (\cA)$ preserves the induced weights in cohomology given by Lemma \ref{lem:homogeneos}.
Now, we consider the cokernel
 $$
 Z^n \buildrel\text{\scriptsize def}\over{:=} \coker \big(H^n(\rho) \colon H^n(\Lambda V^{<n}) \to H^n(\cA)\big).
 $$
By the above, it admits a basis $\{z_i\}$ formed by $\omega$-homogeneous elements. This implies that
there are cocycles $a_i\in \cA^n$ which are $\omega$-homogeneous and $z_i=[a_i]$.
Consider also the kernel
 $$
 B^n\buildrel\text{\scriptsize def}\over{:=} \ker \big(H^{n+1}(\rho)\colon H^{n+1}(\Lambda V^{<n}) \to H^{n+1}(\cA)\big),
 $$
which has a basis $\{b_j\}$ formed by $\widetilde \omega$-homogeneous elements. Let $\eta_j$ be $\widetilde\omega$-homogeneous
elements of $\Lambda^{n+1} V^{<n}$ with $b_j=[\eta_j]$. As $H^{n+1}( \rho) (b_j)=0= [\rho(\eta_j)]$, we have
that $\rho(\eta_j)=d \psi_j$ for some $\psi_j\in \cA$. We can assume that  $\psi_j$ is $\omega$-homogeneous and $\omega(\psi_j)=
\widetilde\omega(\eta_j)$.

Now we define
 $$
 V^n\buildrel\text{\scriptsize def}\over{:=} Z^n\oplus B^n,  \qquad dz_i=0, \, db_j=\eta_j,
 $$
declare $\widetilde\omega(z_i)=\omega(a_i)$, $\widetilde\omega(b_j)=\omega(\psi_j)$
and construct the morphism
 $$
 \rho\colon V^n \to \cA
 $$
given by $\rho(z_i)=a_i$ and $\rho(b_j)=\psi_j$.

By the inductive hypothesis $H^{<n}(\rho)$ is an isomorphism and since $H^n(\Lambda V^{<n+1})=H^n(\Lambda V^{<n})\oplus Z^n$ by construction, then $H^n(\rho)$ is an isomorphism too. Moreover, as we are considering simply-connected dgas, $H^{n+1}(\Lambda V^{<n+1})=H^{n+1}(\Lambda V^{<n})/B^n$ again by construction, hence  $H^{n+1}(\rho)$ is injective, which finishes the inductive step.
\end{proof}

\begin{lemma}\label{lem:ideal_homogeneo}
Let $(\mathcal{A},d)$ be a connected dga with a positive weight $\omega$, and $I\subset \mathcal{A}$
be a differential closed ideal (that is, $dI\subset I$)
generated (as a vector space) by $\omega$-homogeneous elements. Then $\widetilde{\mathcal{A}}=\mathcal{A}/I$ is a connected dga with positive weight $\widetilde{\omega}$ defined by $\widetilde{\omega}(\bar{a})=\omega(a)$, for every $\omega$-homogeneous element $a\in\mathcal{A}$ such that $\bar{a}\ne 0$.
\end{lemma}

\begin{proof}
Since $I\subset \mathcal{A}$ is a differential closed ideal,
then $\widetilde{\mathcal{A}}=\mathcal{A}/I$ is a connected dga, and it only remains to prove that the weight $\widetilde{\omega}$ is well defined.

Assume $\widetilde{\omega}$ is not well defined. Therefore there exist $\omega$-homogeneous elements $a_l\in\mathcal{A}$, $l=1,2$, such that $\omega(a_1)\ne\omega(a_2)$ and $\bar{a}_1=\bar{a}_2\ne 0$.  Then $a_1-a_2\in I$, and $a_1-a_2=\sum\limits_{i=0}^r x_i$ where every $x_i\in I$ is $\omega$-homogeneous, and $\omega(x_i)=\omega(x_j)$ if and only if $i=j$. Moreover, $a_l$ is $\omega$-homogeneous, $l=1,2$, and $\omega(a_1)\ne\omega(a_2)$.
Hence $a_l=x_{i(l)}$, $l=1,2$, and $\bar{a}_1=\bar{a}_2= 0$. This is a contradiction as we assumed $\bar{a}_1=\bar{a}_2\ne 0$.
\end{proof}

Finally, we obtain an integral version of Corollary \ref{cor:44}.

\begin{theorem}\label{thm:weigth_implies_integral_flexible}
Let $M \in \operatorname{Mfd}_N$, $N\geq 4$,  with minimal model $(\Lambda V,d)$ and let $\eta$ be its cohomological fundamental class. Write the rational cohomological fundamental class as $\eta \otimes_\mathbb Q 1=[\nu]$.
Assume there exists a dga morphism $\psi\colon (\Lambda V,d)\to (\mathcal{A},d)$ where $(\mathcal{A},d)$ is a simply-connected finite type dga with positive weight and $H^{N}(\psi)([\nu]) \ne 0$. Then $M$ is not strongly inflexible.
\end{theorem}

\begin{proof}
 By Lemma \ref{lem:homogeneos}, $H^N(\psi) ([\nu]) =\sum\limits_{i=0}^r [a_i]$ where every $a_i\in \mathcal{A}$ is an $\omega$-homogeneous cocycle, and $\omega(a_i)=\omega(a_j)$ if and only if $i=j$. Fix $\widetilde{a}\in \mathcal{A}$, a nontrivial $a_i$ in the decomposition above,
 take a complement $\mathcal{A}^N=\la \widetilde{a}\ra\oplus W$, where $W$ is spanned by $\omega$-homogeneous elements, and define
 $$
 I\buildrel\text{\scriptsize def}\over{:=}\mathcal{A}^{\geq N+1} \oplus W.
$$
Then $I \subset\mathcal{A}$ is a closed differential ideal generated (as a vector space) by $\omega$-homogeneous elements.
By Lemma \ref{lem:ideal_homogeneo}, $\widetilde{\mathcal{A}}=\mathcal{A}/I$ is a finite type connected dga with positive weight and formal dimension $N$.
Therefore, by Proposition  \ref{prop:pesos_en_modelos},  $\widetilde{\mathcal{A}}$ admits a minimal model  $(\Lambda W, d)$ with positive weight, which is the rational homotopy type of a simply-connected finite CW-complex $X$, with $H^N(X;\mathbb{Q})\cong\mathbb{Q}$. Moreover $X$ is universal by Theorem \ref{thm:universal-weights}.

Let us consider the composition
$
 \xymatrix{
 {\widetilde{\psi}}\colon \Lambda V\ar[r]^-\psi & \mathcal{A} \ar@{>>}[r]& \widetilde{\mathcal{A}},
 }
$
and  let $\widetilde{\Psi}\colon X_\Q\to M_\Q$ be its geometrical realization.
Observe that $H^N(\widetilde{\psi}) ([\nu])\neq 0$ and so $H^N(\widetilde{\Psi}) ([\nu])\neq 0.$

As $X$ is universal, for $\zeta_M\colon M \rightarrow M_\Q$  the rationalization of $M$ and $\zeta_X\colon X \rightarrow X_\Q$ the rationalization of $X$, there exists a commutative diagram
\begin{equation*}
\xymatrix{
X \ar@{.>}[d]_{g}\ar@{.>}[r]^f & X\ar[d]^{\widetilde{\Psi} \circ\zeta_X}\\
M\ar[r]^{\zeta_M} & M_\Q \\
}
\end{equation*}
where $f$ is a $0$-equivalence. We are going to show that $H^N(g; \Q)([\nu] )\neq 0$ and therefore, by Theorem \ref{thm:universal_implies_flexible}, $M$ is not strongly inflexible:
$$
 H^N ( g \,; \mathbb Q) ([\nu]) = H^N ( \widetilde{\Psi} \circ \zeta_X \circ f  \,; \, \mathbb Q) ([\nu]) = H^N( f; \, \mathbb Q)(H^N(\widetilde{\Psi}) ([\nu]))
 \ne 0.
 $$
\end{proof}

\section{$3$-step Sullivan algebras}\label{sec:3-step}

Positive weights in Sullivan algebras relate to non inflexibility and non-strong inflexibility properties.
 We are going to exploit this aspect of positive weights to give a systematic way to check that certain 3-step Sullivan algebras are not strongly inflexible.
\begin{definition}\label{def:3step} We say that a Sullivan algebra
 $$
 \cM=(\Lambda V_1\otimes \Lambda V_2\otimes \Lambda V_3,d),
 $$
is a $3$-step algebra if $d(V_2)\subset \Lambda V_1$ and $d(V_3) \subset \Lambda (V_1\oplus V_2)$.
\end{definition}

We are only interested in $3$-step Sullivan algebras of the form
$$
 \cM=(\Lambda V_1\otimes \Lambda V_2\otimes \Lambda (z),d).
 $$
For such Sullivan algebras, we are going to describe how to construct a  universal dga $\cB_\cM$ admitting positive weights,  and a morphism of dgas $\psi\colon \cM \rightarrow \cB_\cM.$ Therefore, we have a sufficient condition, Corollary \ref{cor:44},  to check that $\cM$ is not strongly inflexible.

We first introduce some notions.
\begin{definition}\label{def:derivation_algebra}
A derivation differential graded algebra, ddga for short, is a triple $(\cA,d,\theta)$ such that $(\cA,d)$ is a dga,
and $\theta$ is a derivation of degree $0$ on $\cA$ such that  $d\theta=\theta d$.

Let $(\cA,d,\theta)$ and $(\cA',d',\theta')$ be ddga. A ddga morphism $f\colon \cA\to \cA'$ is a dga morphism
such that $f\circ \theta=\theta'\circ f$.
\end{definition}

\begin{definition}\label{def:free-ddga}
Let $(\cA,d)$ be a dga. We call the free derivation differential graded algebra generated by $A$ to a ddga $(P(\cA),d,\theta)$ equipped with a dga morphism  $i\colon\cA \to P(\cA)$ which is initial for ddgas morphisms, that is,
for any ddga $(\cA',d',\theta')$ and any dga map $f\colon\cA\to \cA'$, there is a unique ddga morphism
$\bar{f}\colon P(\cA)\to \cA'$ such that $f=\bar{f}\circ i$.
\end{definition}

Clearly, as it is defined by a universal property, if it exists, $(P(\cA),d,\theta)$ has to be unique up to isomorphism.
For its existence, it can be constructed as follows. Let
 $$
  \widehat{P}(\cA)=T(\cA\oplus \theta \cA\oplus \theta^2 \cA\oplus \ldots)
   $$
be the tensor algebra on elements of $\cA$ and ``formal'' derivatives $\theta^k\cA\cong \theta(\theta^{k-1})\cA$, for $k\geq 1$,
and $\theta ^0\cA=\cA$.
Consider
$I\subset \widehat{P}(\cA)$ the ideal generated by elements $a\otimes b -ab$, and
$\theta (ab)-\theta a \otimes b-a\otimes \theta b$, for any pair of elements $a,b\in \cA$. This is a differential ideal,
that is $\theta(I)\subset I$. Finally, set
 $$
 P(\cA)=\widehat{P}(\cA)/I,
 $$
 and $d(\theta a)=\theta(da)$, for all $a\in \cA$.

Unfortunately, the algebra $P(\cA)$ is not of finite type. To solve this, we define the dot algebra.

\begin{definition}\label{def:dot_algebra}
Let $(\cA,d,\theta)$ be a ddga. We say $\cA$ is \emph{dot algebra} if $\theta^2=0$.
\end{definition}

There is a notion of free dot algebra generated by a dga:

\begin{definition}\label{def:free_dot_algebra}
Let $(\cA,d)$ be a dga. We call the \emph{free dot algebra generated by $\cA$}
or simply  dot algebra of $\cA$, to a dot algebra $(\dot{\cA},d,\theta)$ equipped
with a dga morphism $i\colon \cA \to \dot{\cA}$ which is universal for dot algebras, that is
for any dot algebra $(\cA',d',\theta')$ and any dga map $f\colon \cA\to \cA'$, there is a unique ddga morphism
$\bar{f}\colon  \dot{\cA}\to \cA'$ such that $f=\bar{f}\circ i$.
\end{definition}

The concrete construction of a model is
 $$
  \dot{\cA}=P(\cA)/P^{\geq 2}(\cA),
  $$
 which is the quotient of $P(\cA)$ by the ideal generated by $\theta^j(\cA)$, $j\geq 2$,  and $\theta(\cA)\cdot \theta(\cA)$.
 We denote
  $$
  \dot a=\theta (a).
  $$

Now let us deal with the situation of  $3$-step Sullivan algebras
\begin{equation}\label{def:3stepz}
 (\cM,d) = (\Lambda V_1\otimes \Lambda V_2\otimes \Lambda (z),d)
\end{equation}
 Let $\cM_{[2]}\buildrel\text{\scriptsize def}\over{:=}\Lambda V_1\otimes \Lambda V_2=\Lambda (x_i,y_j)$ be the $2$-step dga associated to  \eqref{def:3stepz}, where the generators of $V_1$ are denoted by $x_i$, and the generators of $V_2$
 by $y_j$. We endow $(\cM_{[2]}, d)$  with the  positive weight $\omega$  given by Proposition \ref{prop:55}.
 Hence, we can decompose  $dz \in \cM_{[2]}$ into $\omega$-homogeneous elements:
 $$
  dz=P_0+P_1+\ldots +P_m
  $$
Note that elements $P_0,P_1,\ldots, P_m$ are all cocycles.
If $m=0$, then
$dz$ is $\omega$-homogeneous and we can define $\omega(z)=\omega(P_0)$. Then $ \cM$
has a positive weight and therefore applying Corollary \ref{cor:44} to $\varphi= id_\cM$ we obtain that $\cM$ is not strongly inflexible.

We focus on the case $m=1$, which is all that we need for our applications. So, we shall assume that $dz=P_0+P_1$. Then, the dot algebra of  $(\cM_{[2]}, d)=(\Lambda (V_1\oplus V_2), d)$ is actually
 \begin{equation}\label{eqn:dotcA}
  (\dot \cM_{[2]}, d)=(\cM_{[2]} \otimes (\mathbb Q \oplus \dot V_1\oplus \dot V_2), d)
  \end{equation}
  where the new elements $\dot x_i, \dot y_j$ have differentials given by:
\begin{align*}
 d \dot x_i &=  0, \\
 d \dot y_j &= \dot {d y_j}\\
 \end{align*}
 We introduce the following dga:
   \begin{equation}\label{def:universal}
(\cB_{\cM}, d) \buildrel\text{\scriptsize def}\over{:=}(\dot \cM_{[2]} \otimes \Lambda (u_1,u_2,u_3), d)
\end{equation}
where the new elements $u_1,u_2,u_3$ have differentials given by:
\begin{align*}
 du_1 &= P_0+\dot P_1, \\
 du_2 &= P_1, \\
 du_3 &= \dot P_0 .
 \end{align*}

We assign a weight $\omega$ to  $ \dot \cM_{[2]}$ given by 
 $$
 \omega(\dot v)=\omega(v) +(\omega(P_0) -\omega(P_1)),
 $$
for $v\in V_1\oplus V_2$. That way,  $P_0+\dot P_1$ is $\omega$-homogeneous and
we can extend $\omega $ to a weight in $(\cB_{\cM}, d)$ as follows:
 $$
  \omega(u_1)=\omega(P_0+\dot P_1), \, \omega(u_2)=\omega(P_1), \,\omega(u_3)=\omega(\dot P_0).$$

\begin{remark}\label{rem:positiveB}Observe that if  $(\cM,d)$ in \eqref{def:3stepz} is $(\omega(P_1) -\omega(P_0))$-connected, then $\omega$ is a positive weight in $(\cB_\cM,d)$.
\end{remark}

 Finally consider the dga morphism
 \begin{equation}\label{eqn:psi}
 \begin{array}{rcl}
 \psi\colon \cM=\cM_{[2]} \otimes \Lambda (z) & \longrightarrow & \cB_\cM=\dot \cM_{[2]} \otimes \Lambda (u_1,u_2,u_3) ,\\
  m\in \cM_{[2]} & \mapsto & m+\dot m,\\
  z & \mapsto & u_1+u_2+u_3.
  \end{array}
 \end{equation}

The following result gives  a criteria to decide when the morphism $\psi$ above detects the fundamental class of the $3$-step Sullivan algebra $\cM$ with Poincar{\'e} duality cohomology.

\begin{theorem}\label{thm:cor:main1}
Let  $(\cM,d) = (\cM_{[2]}\otimes \Lambda (z),d)$ be a $3$-step Sullivan algebra with Poincar\'e duality cohomology, where  $\cM_{[2]}$ is $2$-step and endowed with the positive weight $\omega$ given by Proposition \ref{prop:55}, and $d(z)=P_0+P_1\in\cM_{[2]}$ where $P_0$ and $P_1$ are $\omega$-homogeneous. Let $[\nu]\in H^N(\cM)$, $N\geq 4$, be the fundamental class of $\cM$ and assume that the representative $\nu\in\cM_{[2]}$. If the following holds
\begin{enumerate}
\item\label{thm:cor:main1-1} $\cM$ is $(\omega(P_1) -\omega(P_0))$-connected,
\item\label{thm:cor:main1-2} There are not any nontrivial $\xi\in H^{N+1-2|dz|}(\cM_{[2]})$, such that $\xi\cdot[P_0]=0= \xi\cdot [P_1]$, and
\item\label{thm:cor:main1-3} For $A,B \in \cM_{[2]}^{N-|dz|}$ cocyles such that $[\nu] =[AP_0+B P_1]$, then $(A-B) \dot P_1$ (up to a coboundary) is not in the ideal of $\dot \cM_{[2]}$ generated by $P_0,\dot P_0, P_1$,
\end{enumerate}
then $H^N(\psi)([\nu]) \neq 0$, where $\psi$ is the morphism \eqref{eqn:psi}. In particular $\cM$ is not strongly inflexible. Moreover, if $M$ is a simply-connected $N$-manifold for which  $\cM$ is a model, then $M$
is not strongly inflexible either.
\end{theorem}

\begin{proof}
Suppose that $H^N(\psi)([\nu])=0$. Then
 $\nu+\dot \nu=d \chi$ in $\cB_\cM$ and we can write
  $$
  \chi =\Big(\sum \limits_{\substack{1\leq i <j  \leq  3}} (A_i + \dot B_i) u_i + (C_{ij}+\dot D_{ij}) u_i u_j  \Big)+ (E+\dot F) u_1u_2u_3+ (G+\dot H),
  $$
  for some elements $A_i,\dot B_i, C_{ij}, \dot D_{ij},E,\dot F, G, \dot H$ in $\dot \cM_{[2]}$, where the dot means that
the element lies in the dot part, and the indexes $i<j$.

Let $n=N-|dz|$. Then we have the following set of equations
\begin{equation}\label{eqn:list}
 \begin{aligned}
  \nu & = (-1)^n(A_1P_0+ A_2P_1)+ dG ,\\
  \dot\nu &= (-1)^n(\dot B_1 P_0+ A_1\dot P_1+ \dot B_2 P_1+ A_3\dot P_0)+d\dot H, \\
  0 &=dA_1+(-1)^n C_{12}P_1, \\
  0 &=dA_2+ (-1)^{2n-N+1}C_{12}P_0, \\
  0 &= dC_{12} ,
    \end{aligned}
    \end{equation}
  the last three obtained by taking the dot--free coefficients of $u_1$, $u_2$, and $u_1u_2$, respectively, in the equality $\nu+\dot \nu=d \chi$.

 From the above, $\xi=[C_{12}]$ is a cohomology class
  with  $\xi \cdot [P_0]=0= \xi \cdot [P_1]$. By hypothesis \eqref{thm:cor:main1-2}, $C_{12}=d\eta$ for some $\eta$. Hence
\begin{align*}
A &=A_1-(-1)^n \eta P_1,\\
B &=A_2-(-1)^{2n-N+1}\eta P_0,
\end{align*}
are cocycles and $\nu= (-1)^n(A P_0+BP_1)+ dG$.

Now, in \eqref{eqn:list},  apply the dot operator to the first equation  to get
 $$
  \dot\nu= (-1)^n(\dot A_1P_0+ \dot A_2P_1+ A_1\dot P_0+ A_2\dot P_1)+d\dot G,
  $$
and compare it with the second equation. Then, we obtain 
that $(A_1-A_2)\dot P_1$ is up to a coboundary, in the ideal of  $\dot \cM_{[2]}$ generated by  $P_0,\, \dot P_0$ and $P_1$. Clearly $(A-B)\dot P_1$ also lies
in the same ideal, which is a contradiction with hypothesis \eqref{thm:cor:main1-3}. Hence, $H(\psi)([\nu]) \neq 0$.

Finally, Remark \ref{rem:positiveB} and hypothesis \eqref{thm:cor:main1-1} guarantee that the weight $\omega $ in $\cB_\cM$ is  a positive weight, which implies that $\omega$ detects $H(\psi)([\nu])$. Applying Corollary \ref{cor:44} we conclude that $\cM$ is not strongly inflexible. The last statement follows from  Theorem \ref{thm:weigth_implies_integral_flexible}.
\end{proof}

\section{The Arkowitz-Lupton Example} \label{sec:first}

The existence of inflexible manifolds was first established by Arkowitz and Lupton in \cite[Example 5.1, Example 5.2]{AL2}. They gave examples of simply-connected Sullivan algebras with Poincar\'e duality cohomology that have finitely many homotopy classes of dga endomorphisms. Then, using Barge and Sullivan obstruction theory, they showed that those examples are minimal models of simply-connected manifolds. In particular, these manifolds are inflexible.

The succeeding examples in literature have been built upon \cite[Example 5.1]{AL2} that we now introduce:

\begin{definition}{(The Arkowitz-Lupton example)}\label{def:dga_AL} Let
$(\cM, d)=(\Lambda (x_1,x_2,y_1,y_2,y_3,z),d)$ be the Sullivan algebra with
 \begin{align*}
 |x_1|=8 \qquad & dx_1=0 \\
 |x_2|=10 \qquad & dx_2=0 \\
 |y_1|=33 \qquad & dy_1=x_1^3x_2 \\
 |y_2|=35 \qquad & dy_2=x_1^2x_2^2 \\
 |y_3|=37 \qquad & dy_3=x_1x_2^3 \\
 |z|=119 \qquad & dz=x_1^4\gamma+x_1^{15}+x_2^{12}
 \end{align*}
where $\gamma=\alpha\beta$, $\alpha=x_1y_2-x_2y_1$, $\beta=x_1y_3-x_2y_2$.
So $\gamma=x_1^2 y_2y_3-x_1x_2y_1y_3 +x_2^2 y_1y_2$.
\end{definition}

We are going to show that the dga from Definition \ref{def:dga_AL} is not strongly inflexible. Observe that it is a $3$-step algebra as we introduced in the previous section. With the same notation as in Section \ref{sec:3-step},  the $2$-step algebra
$$(\cM_{[2]}, d)=(\Lambda (x_1,x_2,y_1,y_2,y_3),d),$$
has a positive weight $\omega$ given by Proposition \ref{prop:55}.
 The cohomology is
  \begin{align*}
 H(\cM_{[2]}) =& {\mathbb Q}\{ x_1^n, x_2^m, x_1^n \alpha, x_2^m \alpha, x_1^n \beta, x_2^m \beta,
  x_1^n \gamma, x_2^m \gamma \, |\,  n,m\geq 0\} \\
  &\oplus {\mathbb Q}\{ x_1x_2, x_1^2x_2,x_1x_2^2, x_1x_2\alpha, x_1x_2\beta,
  x_1x_2^2\alpha \}.
  \end{align*}
 Notice that there is a ``free'' part (first summand) and some extra low degree generators.
 There are the following coboundaries:
 \begin{equation}\label{eqn:yy}
 d(y_1y_2)=x_1^2x_2\alpha, \; d(y_2y_3)=x_1x_2^2\beta, \; d(y_1y_3)=x_1^2x_2\beta+ x_1x_2^2\alpha.
 \end{equation}
We also introduce the dot algebra associated to $\cM_{[2]}$, see \eqref{eqn:dotcA},  $$(\dot{\cM_{[2]}},d)\ =(\cM_{[2]} \otimes (\mathbb Q \oplus \dot V_1\oplus \dot V_2), d) \, \text{where}$$
 \begin{align*}
 |\dot x_1|=8 \qquad & d\dot x_1=0 \\
 |\dot x_2|=10 \qquad & d\dot x_2=0 \\
 |\dot y_1|=33\qquad & d\dot y_1=3x_1^2 \dot x_1 x_2  + x_1^3 \dot x_2 \\
 |\dot y_2|=35 \qquad & d\dot y_2=2x_1 \dot x_1 x_2^2 + 2x_1^2 x_2 \dot x_2\\
 |\dot y_3|=37 \qquad & d\dot y_3=\dot x_1 x_2^3 +3x_1x_2^2 \dot x_2
 \end{align*}
 and  $\dot\gamma=\dot\alpha \beta+\alpha\dot\beta$, with
$\dot\alpha= \dot x_1 y_2+x_1 \dot y_2-\dot x_2y_1-x_2\dot y_1$,
$\dot\beta=\dot x_1y_3+x_1\dot y_3-\dot x_2y_2-x_2\dot y_2$.
Therefore
 \begin{align}\label{eq:dotgamma}
 \dot\gamma= &\, 2 \dot x_1 x_1  y_2y_3- \dot x_1x_2y_1y_3 -x_1\dot x_2 y_1y_3 +2 x_2 \dot x_2 y_1y_2  \nonumber\\
&+x_1^2 \dot y_2y_3-x_1x_2\dot y_1y_3 +x_2^2 \dot y_1y_2+
x_1^2 y_2\dot y_3-x_1x_2y_1\dot y_3 +x_2^2 y_1\dot y_2
 \end{align}

We have gathered the necessary elements to apply Theorem \ref{thm:cor:main1} and prove our main result in this section.  For the proof to be more readable, we have moved to the end of this section some technical lemmas.

\begin{theorem}\label{thm:dga_lowdegrees}
 The dga $\cM$ from Definition \ref{def:dga_AL} is not strongly inflexible. Furthermore, a manifold $M$ for which $\cM$ is a Sullivan model is not strongly inflexible either.
 \end{theorem}

\begin{proof}
 The cohomology $H(\cM)$ is a Poincar\'e duality algebra (\cite[5.3 Remarks]{AL2}) where  $\nu=x_1^{26} \in \cM_{[2]}$ is  the representative of the fundamental class $[\nu] \in H^{208} (\cM)$, and
 $dz=P_0+P_1$, where
  \begin{align*}
  P_0 &= x_1^{15}+x_2^{12} ,\\
  P_1 &= x_1^4\gamma .
  \end{align*}
We now check that hypothesis \eqref{thm:cor:main1-1} -- \eqref{thm:cor:main1-3} in Theorem \ref{thm:cor:main1} hold.

It is clear that hypothesis \eqref{thm:cor:main1-1} holds since $(\omega(P_1) -\omega(P_0))=122-120=2$, and $\cM$ is $7$-connected. Moreover, $N+1-2|dz|=208+1-240=-31<0$ hence hypothesis \eqref{thm:cor:main1-2} holds too.

We are now going to check hypothesis \eqref{thm:cor:main1-3}. For $N-|dz|=208-120 =88$, we know that $H^{88}(\cM_{[2]})={\mathbb Q}\{ [x_1^{11}], [\gamma]\}$.
 Let us write $[\nu]=[A P_0+ B P_1]$, where $A,B\in \cM_{[2]}^{88}$ are cocycles. The possible
cases are
  \begin{align*}
 A &= x_1^{11} +d \eta_A, \\
 B &= a \gamma + d \eta_B,
 \end{align*}
 for $a\in \Q$.  Then, we deduce that
  $$
  (A-B)\dot P_1=\big(4x_1^{14} \dot x_1 \gamma + x_1^{15} \dot \gamma \big) +d((\eta_A-\eta_B)\dot{P_1})
  $$
using that $\dot P_1= 4x_1^3 \dot x_1 \gamma + x_1^4 \dot \gamma$, $\gamma^2=0$ and $\gamma \dot\gamma=0$.
In order to apply Theorem \ref{thm:cor:main1}, we need to check that the element $(A-B)\dot P_1$ is not in the ideal
 $$
 I= \la x_1^{15}+x_2^{12} , 15 x_1^{14}\dot x_1+12x_2^{11}\dot x_2, x_1^4\gamma \ra+\im d \subset \dot\cM_{[2]}.
  $$
Since $4x_1^{14} \dot x_1 \gamma  \in I,$ it is enough to check that $X \buildrel\text{\scriptsize def}\over{:=}x_1^{15}\dot\gamma\notin I$.

First, notice that there is a $y$-gradation in $\dot\cM_{[2]} $ according to the number of $y_j,\dot y_j$, and the differential decreases the
degree by one. The  $y$-degree of $X$ is two, so let us assume that $X=Y+ d \eta$,  for  some cocycle $Y$ in $I$ of $y$-degree equals two. Namely
\begin{equation}\label{eqn:Y_AL}
 Y=\dot F (x_1^{15}+x_2^{12}) + G(15 x_1^{14}\dot x_1+12x_2^{11}\dot x_2 ) +\dot Hx_1^4 \gamma,
 \end{equation}
 where $\dot F,  \, G, \dot H$ are of the
form
 \begin{align*}
 \dot F &= \dot F'+ \dot F'', \, \, \dot F'=\sum \limits_{\substack{1\leq i , j  \leq  3}} F'_{ij} y_i\dot y_j , \,\, \dot F''= \sum \limits_{\substack{1\leq i < j  \leq  3}} \Big( \sum \limits_{1 \leq k \leq 2} F''_{ijk}\dot x_k y_iy_j \Big), \\
 G &=\sum \limits_{\substack{1\leq i < j  \leq  3}}G_{ij} y_iy_j, \\
 \dot H &= \sum  \limits_{1 \leq k \leq 2} H_k\dot x_k,
 \end{align*}
with $F'_{ij}, \,  F''_{ijk}, \,G_{ij}, \, H_k \in \mathbb Q [x_1, x_2].$ As $Y$ is a cocycle, we have
 \begin{equation*}
 0 =dY=d(\dot F'+ \dot F'') ( x_1^{15}+x_2^{12}) + dG( 15 x_1^{14}\dot x_1+12x_2^{11}\dot x_2 ).
 \end{equation*}
Look at the $\dot y_j$-term in this expression, $j=1,2$, to conclude that $\sum \limits_{1 \leq i \leq 3}  F _{ij}'y_i$ is a cocycle. Hence,
according to degrees, we deduce that
\begin{equation*}\label{eqn:wi_AL}
	\dot F' \in W={\mathbb Q}\{  x_2\beta\dot y_1, x_2\alpha \dot y_2, x_1\beta \dot y_2,
x_1\alpha \dot y_3\}.
\end{equation*}

Recalling that $X=Y +d\eta$, and (\ref{eqn:yy}), we see that the $\dot y$-part
of the exact term $d\eta$  is of the form $x_1x_2 \sum \limits_{1 \leq j \leq 3}  (R_j\alpha+S_j\beta) \dot y_j$, with $R_j, S_j$ elements in $\mathbb Q[x_1, x_2].$

Looking at the components $x_1\alpha,  x_2\alpha,x_1\beta,x_2\beta$, and comparing the $\dot y_j$-parts, $j=1,2,3$ of $X$, $Y$ and $d\eta$, we obtain that:
  $$
  \dot F'=   x_1\beta\dot y_2 +x_1\alpha\dot y_3\, .
  $$

Therefore using \eqref{eq:dotgamma} and \eqref{eqn:Y_AL}, the equality   $X=Y+d\eta$ becomes
 \begin{align*}
  &x_1^{15}\dot\gamma = x_1^{15}\big(-\beta x_2 \dot y_1 +(\beta x_1-\alpha x_2) \dot y_2 + \alpha x_1\dot y_3\big)
   +  x_1^{15} \big(  (\beta y_2 + \alpha y_3) \dot x_1-(\beta y_1+\alpha y_2 ) \dot x_2\big)
  = \\
 &=  (  x_1\beta\dot y_2 +x_1\alpha\dot y_3) (x_1^{15}+x_2^{12})
 + \dot F'' (x_1^{15}+x_2^{12}) + G( 15 x_1^{14}\dot x_1+12x_2^{11}\dot x_2) +\dot H x_1^4\gamma +d\eta,
 \end{align*}
which can be rewritten as
 \begin{align*}
   \dot F'' &(x_1^{15}+x_2^{12}) + G( 15 x_1^{14}\dot x_1+12x_2^{11}\dot x_2) +\dot H x_1^4\gamma
 -  x_1^{15} \big(  (\beta y_2 + \alpha y_3) \dot x_1-(\beta y_1+\alpha y_2 ) \dot x_2 \big)
    = \\
 &= -x_1^{15} x_2 \beta \dot y_1 - x_1^{15} x_2 \alpha\dot y_2
-   x_1x_2^{12}\beta \dot y_2-x_1x_2^{12} \alpha \dot y_3 - d\eta\, .
 \end{align*}

Now use the formulas (\ref{eqn:yy}) to get rid of the terms $\dot y_j$ at the expense of exact
terms, for example $x_1^{15} x_2\beta \dot y_1= d(x_1^{12}y_1\beta \dot y_1) -
x_1^{12}y_1\beta d \dot y_1$. Hence
 \begin{align}\label{eq:nodoty}
    \dot F'' &(x_1^{15}+x_2^{12}) + G( 15 x_1^{14}\dot x_1+12x_2^{11}\dot x_2) +\dot H x_1^4\gamma
 -  x_1^{15} \big(  (\beta y_2 + \alpha y_3) \dot x_1-(\beta y_1+\alpha y_2 ) \dot x_2 \big)   = \nonumber \\
 &=  x_1^{12}y_1 \beta d \dot y_1  +x_1^{12}  y_1\alpha d \dot y_2 +
      x_2^9 y_3 \beta d\dot y_2 + x_2^9y_3 \alpha d\dot y_3 + d\tilde\eta\, ,
 \end{align}
 which is an equation in $ \cM_{[2]} \otimes {\mathbb Q}\{ \dot x_1,\dot x_2\}$. Using the formulas for $d\dot y_1,d\dot y_2,d\dot y_3$
 and separating the components in $\dot x_1,\dot x_2$ independently, we get two equations in $\cM_{[2]}$:
 \begin{align*}
   F''_1 &(x_1^{15}+x_2^{12}) + 15 Gx_1^{14} +H_1 x_1^4\gamma -  x_1^{15}  (\beta y_2 + \alpha y_3)+d\tilde{\eta}_1 = \\
 &= 3 x_1^{14} x_2 y_1\beta   +2x_1^{13} x_2^2 y_1\alpha
  + 2x_1  x_2^{11} y_3  \beta+ x_2^{12}  y_3\alpha \, , \\
  F''_2 & (x_1^{15}+x_2^{12})+12 Gx_2^{11} +H_2 x_1^4 \gamma + x_1^{15}(\beta y_1+\alpha y_2 )   +d\tilde{\eta}_2 = \\
 &= x_1^{15} y_1\beta    +2x_1^{14} x_2 y_1\alpha
 +  2 x_1^2  x_2^{10} y_3\beta  +3x_1 x_2^{11} y_3 \alpha.
 \end{align*}

We look at the first equation modulo $\gamma, x_2^{12}$. By Lemma \ref{lem:dotF''andG}, $G = e\gamma$ for some  $e \in \mathbb Q$, and by Lemma \ref{lem:dtildeeta} and Remark \ref{rem:dtildeeta}, $d\tilde{\eta}_1$ is in the ideal generated by $\gamma$. Hence, the first equation reduces to
 \begin{equation}\label{eqn:vm}
   F''_1 (x_1^{15}) -  x_1^{15}  (\beta y_2 + \alpha y_3)
 = 3 x_1^{14} x_2 y_1\beta   +2x_1^{13} x_2^2 y_1\alpha
  + 2x_1  x_2^{11} y_3  \beta.
  \end{equation}
From Lemma \ref{lem:dotF''andG} we know that the $\dot x_1$-part of $\dot F''$ is $F''_1= x_2y_1y_3 + x_1y_2y_3$. Then, the left hand side of
(\ref{eqn:vm}) is of the form
 $$
 2 x_1^{15} x_2 y_1 y_3 + x_1^{16}y_2y_3\, .
 $$
The right hand side of (\ref{eqn:vm}) modulo $\gamma, x_2^{12}$ reduces to
 $$
 -3 x_1^{15} x_2 y_1 y_3 + 5x_1^{14}x_2^2y_1y_2\, .
 $$
 Hence, we obtain that
 $$
  x_1^{16}y_2y_3 + 5 x_1^{15} x_2 y_1 y_3  - 5 x_1^{14}x_2^2y_1y_2 = 0  \mod (\gamma, x_2^{12}),
 $$
which by Lemma \ref{lem:contradiction} leads to a contradiction. Therefore, the element $X$ is not in the ideal $I$ and hypothesis \eqref{thm:cor:main1-3} holds.

Since we have proved that hypothesis \eqref{thm:cor:main1-1} -- \eqref{thm:cor:main1-3} hold, we conclude that neither $\cM$ from Definition \ref{def:dga_AL}  nor any manifold $M$ for which $\cM$ is a Sullivan model are strongly inflexible.
   \end{proof}

 \begin{example}(Inflexible dgas in literature)\label{ex:list_dga}
We list the collection of  inflexible dgas in literature that follow the same pattern as Arkowitz-Lupton's example. In all cases, they are  inflexible $3$-step dgas $\cM=(\Lambda (x_1,x_2,y_1,y_2,y_3,z),d)$. The  sequence of degrees $\{|x_1|, |x_2|, |y_1|, |y_2|, |y_3|, |z|\}$ and nontrivial differentials are given in the tables below:
 \begin{equation*}\label{table:inflexible_1}
\begin{array}{|c|c|c|c|}
\hline
&\text{\cite[Ex.\ 5.1]{AL2} (cf. \cite[Ex.\ I.3]{CL}})  & \text{\cite[Ex.\ 5.2]{AL2} (cf. \cite[Ex.\ I.4]{CL}})  & \text{\cite[Ex.\ I.1]{CL}}\\ \hline
Degrees & \{8, 10, 33, 35, 37, 119\} & \{ 10, 12, 41, 43, 45, 119\} & \{2, 4, 9, 11, 13, 35\} \\ \hline
dy_1 &x_1^3x_2 & x_1^3x_2 & x_1^3x_2 \\ \hline
dy_2 &x_1^2x_2^2& x_1^2x_2^2  &x_1^2x_2^2\\ \hline
dy_3 &x_1 x_2^3&x_1 x_2^3 & x_1 x_2^3  \\ \hline
dz  &  x_1^4 \alpha\beta +x_1^{15}+x_2^{12} &  x_2 \alpha\beta +x_1^{12}+x_2^{10}  & x_2^2\alpha\beta +x_1^{18}+x_2^{9}  \\ \hline
\nu &  x_1^{26}& x_2^{19}  & x_2^{16} \\ \hline
\end{array}
\end{equation*}

\begin{equation*}\label{table:inflexible_2}
\begin{array}{|c|c|c|c|}
\hline
& \text{\cite[Ex.\ I.2]{CL}} &  \text{ \cite[Def.\ 1.1]{CosMenVir18}} \;  \,k \geq 1 & \text{\cite[Ex.\ 3.8]{Am}}  \\ \hline
Degrees & \{4, 6, 17, 19, 21, 59\} &
\parbox[c][1cm]{5.5cm}
{\tiny
$\{ 10k+8,$ $12k+10$, $42k+33$, $44k+35$, $46k+37$, $60k^2+38k+39\}$
}
&
\{2, 2, 9, 11, 13, 35\}   \\ \hline
dy_1   &x_1^3x_2 &x_1^3x_2 & x_1^3{\bar x_2}  \\ \hline
dy_2 & x_1^2x_2^2 &x_1^2x_2^2& x_1^2{\bar x_2}^2 \\ \hline
dy_3 &x_1 x_2^3 & x_1 x_2^3& x_1 {\bar x_2}^3 \ \\ \hline
dz  & x_2^2 \alpha\beta +x_1^{15}+x_2^{10}  & x_1^{6k-6}\alpha\beta+x_1^{6k+5}+x_2^{5k+4} &{\bar x_2}^2\alpha\beta+x_1^{18}+{\bar x_2}^{9} \\ \hline
\nu  & x_2^{18} & x_1^{6k+16} & x_2^{33}\; {\text{where} \; {\bar x_2}=x_2^2}  \\ \hline
\end{array}
\end{equation*}


\end{example}
The dgas in these  tables are all minimal models of simply-connected inflexible manifolds.  In Theorem \ref{thm:dga_lowdegrees} we have carried out in detail the proof that the dga from Definition \ref{def:dga_AL} (first example of the table) is not strongly inflexible. The other examples can be analogously treated, so following the lines of Theorem \ref{thm:dga_lowdegrees} we obtain:

 \begin{theorem}\label{thm:dga_Tables}
Let $\cM$ be one of the inflexible dgas from Example \ref{ex:list_dga}. Then, $\cM$ is not strongly inflexible.  Furthermore,  any manifold for which $\cM$ is a Sullivan model  is not strongly inflexible either.
 \end{theorem}
\begin{remark}\label{rem:Am-fail}
The dga constructed in
\cite[Section 3]{Am}
is slightly different, since the cohomological fundamental class $\nu$ contains a Massey product like element.
The dga is $(\Lambda (x_1,x_2,y_1,y_2,y_3,y_4,z),d)$ with
 \begin{align*}
 |x_1|=4 \qquad & dx_1=0 \\
 |x_2|=6 \qquad & dx_2=0 \\
 |y_1|=27 \qquad & dy_1=x_1^4x_2^2 \\
 |y_2|=29 \qquad & dy_2=x_1^3x_2^3 \\
 |y_3|=31 \qquad & dy_3=x_1^2x_2^4 \\
  |y_4|=75+4k \qquad & dy_4=x_1^{19+k} \\
 |z|=77 \qquad & dz =x_1x_2\alpha\beta +x_2x_1^{18}+x_2^{13}
 \end{align*}
where $k\geq 0$, and
$\alpha=x_1y_2-x_2y_1$, $\beta=x_1y_3-x_2y_2$, $\delta=x_2^2y_4-x_1^{15+k}y_1$ are nontrivial Massey products.
A representative of the fundamental class is given by
   $\nu=x_2^{26}y_4-x_1^{15+k}x_2^{24}y_1=x_2^{24}\delta.$

We expect that the methods used in this section prove that this
dga is not strongly inflexible. However, showing that the corresponding map \eqref{eqn:psi}
satisfies $H(\psi)([\nu)]) \neq 0$, needs significantly more calculations and our attempts have not proven fruitful yet.
\end{remark}

The end of this section is devoted to prove some technical lemmas that are needed in the proof of Theorem \ref{thm:dga_lowdegrees}.

 \begin{lemma}\label{lem:dtildeeta}
 Let $d\tilde \eta $ be the element in  \eqref{eq:nodoty}. Then, $d\tilde \eta $  is in the ideal generated by the element $\gamma$ introduced in Definition \ref{def:dga_AL}.
 \end{lemma}

 \begin{proof}
From \eqref{eq:nodoty}, we know that $d \tilde \eta  \in  \cM_{[2]} \otimes (\mathbb Q \oplus \dot V_1 )$ and has $y$-length two.   Observe that since $\dot x_i \dot y_j = 0$, we can express $$\tilde \eta = \sum  \limits_{1 \leq i \leq 2} \tilde \eta_i \dot x_i + \sum  \limits_{1 \leq j \leq 3} \hat \eta_j \dot y_j\, ,$$
 where  $\tilde \eta_i, \hat \eta_j\in  \cM_{[2]} $.
 If we apply the differential
 $$d \tilde \eta = \sum \limits_{1 \leq i \leq 2}  d\tilde \eta_i \dot x_i +  \sum \limits_{1 \leq j \leq 3}  d\hat \eta_j \dot y_j \pm \sum \limits_{1 \leq j \leq 3} \hat \eta_j d\dot y_j \, ,$$
we conclude that the elements $d \hat \eta_j=0 $ (since $d \tilde \eta$ does not contain terms with $\dot y_j$) and they are of $y$-length one.  Hence, $\hat \eta_j$ are closed elements of $y$-length two. Reasoning with the degrees of the elements, we obtain that the only possibilities are:
  $$\hat \eta_1 = (a_1 {x_1^2} x_2^7 + b_1 x_1^7 x_2^3)\gamma,$$
   $$\hat \eta_2 = (a_2 {x_1^3} x_2^6 + b_2 x_1^8 x_2^2)\gamma,$$
   $$\hat \eta_3 = (a_3 {x_1^4} x_2^5 + b_3 x_1^9 x_2)\gamma,$$
   for $a_i, b_i \in \mathbb Q$.

 On  the other hand, since the differential decreases by one the $y$-length, $ \tilde \eta$ is of $y$-length three and so 
 $ \tilde \eta_i $ is also of $y$-length three, which implies that  $\tilde \eta_i = R_i(x_1, x_2) y_1y_2y_3.$ Hence
 $$d  \tilde \eta_i  = R_i(x_1, x_2)  d(y_1y_2y_3 )= R_i(x_1, x_2) x_1x_2 \gamma.$$

 Gathering all this information, we deduce that $d \tilde \eta $ is in the ideal generated by $\gamma.$
 \end{proof}

 \begin{remark}\label{rem:dtildeeta}
 Observe that the $\dot x_i$-part of $d \tilde \eta$  is also in the ideal generated by $\gamma$, $i=1,2.$
 \end{remark}

 \begin{lemma}\label{lem:dotF''andG}
 Let $\dot F''$ and $G$ be the terms in \eqref{eqn:Y_AL}. Then $G= e \gamma$, for some $e \in \mathbb Q,$
and $$\dot F'' = x_2 \dot x_2 y_1y_2 + (x_2\dot x_1 + 2x_1 \dot x_2)y_1y_3 + x_1\dot x_1 y_2y_3.$$
  \end{lemma}
  \begin{proof}
 Recall that  $ \dot F''= \sum \limits_{\substack{1\leq i < j  \leq  3}} \Big( \sum \limits_{1 \leq k \leq 2} F''_{ijk}\dot x_k y_iy_j \Big)$
and that $| \dot F'' | = 88.$ By reasoning with degrees, we get:
 $F''_{121} =0, \, F''_{122} =ax_2, \; F''_{131} = bx_2, \; F''_{132} =cx_1, \; F''_{231}= p x_1, \;  F''_{232} =0, \, a, b, c, p \in \mathbb Q,$ so
 $$\dot F'' = ax_2\dot x_2 y_1 y_2 + (bx_2 \dot x_1 + cx_1 \dot x_2) y_1y_3 + p x_1 \dot x_1 y_2 y_3\, .
 $$
 Applying the differential and using Equation \eqref{eqn:yy}, we obtain that
  \begin{align*}
 d\dot F'' =&\,  \dot x_1 \big( -bx_1x_2^4y_1 -px_1^2x_2^3 y_2 + (bx_1^3x_2^7 + px_1^3x_2^2) y_3 \big)\\
 & + \dot x_2 \big( (-ax_1^2x_2^3 - cx_1^2x_2^3)y_1 + ax_1^3x_2^2 y_2 + cx_1^4x_2 y_3 \big).
 \end{align*}

 We proceed in the same way for $G=\sum \limits_{\substack{1\leq i < j  \leq  3}} G_{ij} y_iy_j$ and $| G | = 88.$ By reasoning with degrees, we get
 $G_{12} = ex_2^2$, $G_{13} = fx_1x_2$ and $G_{23} = gx_1^2$, $e, f, g \in \mathbb Q$, so
 $$G = ex_1^2y_1y_2 + fx_1x_2 y_1 y_3 +   gx_1^2 y_2y_3.$$
 Applying the differential, we get that
 $$dG = y_1 (-e-f) x_1^2x_2^4 +y_2 (e-g) x_1^3x_2^3 + y_3 (f+g) x_1^4x_2^2\, .$$

 We now apply the differential to Equation \eqref{eq:nodoty} and compare the $\dot x_i$ components of the new equation for $i=1,2$. Denote $d\dot F''_{\mid \dot x_i}$ the component of $d \dot F''$ in $\dot x_i, $ for $i=1,2$.  We start by comparing the $\dot x_1$-components:
 \begin{align*}
  d\dot F''_{\mid \dot x_1} (x_1^{15} + x_2^{12}) + 15dG  x_1^{14} =&\,
  y_1 (-x_1^{16}x_2^4 -x_1x_2^{16}) +  y_2 (-x_1^{17}x_2^3 -x_1^2x_2^{15})\\
 &+ y_3(2x_1^{18}x_2^2 + 2x_1^3x_2^{14} ),
 \end{align*}
 which is an equation in $\cM_{[2]}.$ Then, using our computation above of $d\dot F''$ and $dG$, we are going to compare the $y_i$-components of this last equation, for $i=1,2,3.$

 Comparing the $y_1$-terms, we get that:
 $$
 -bx_1^{16}x_2^4 - bx_1x_2^{16} - 15 (f+e) x_1^{16}x_2^4 = -x_1^{16} x_2^4-x_1x_2^{16}\, ,
 $$
 so we deduce that $b=1$ and that $e=-f.$ Comparing the $y_2$-terms, we get that $p=1$ and $e=g$.

 We now compare the $\dot x_2$-components of the equation above:
 \begin{align*}
  d\dot F''_{\mid \dot x_2} (x_1^{15} + x_2^{12}) + 12 dG  x_2^{11} =& \,
  y_1 (-3x_1^{17}x_2^3-3x_1^2x_2^{15}) +  y_2 (x_1^{18}x_2^2 -x_1^3x_2^{14})\\
 &+ y_3(2x_1^{19}x_2 + 2x_1^4x_2^{13} ).
 \end{align*}
Again using our computations above of $d\dot F''$ and $dG$ and comparing the $y_i$-terms of this last equation, for $i=1,2,3,$ we obtain that  $a=1$, and $c=2$, which concludes our proof.
  \end{proof}

  \begin{lemma}\label{lem:contradiction}
 The element $Z = x_1^{16}y_2y_3 + 5 x_1^{15} x_2 y_1 y_3  - 5 x_1^{14}x_2^2y_1y_2 $ is not in the ideal generated by $\gamma$ and $x_2^{12}.$
  \end{lemma}
  \begin{proof}
 Let us suppose that  $Z \in \la\gamma, x_2^{12}\ra $. Since $|Z| = 200$, by reasoning with degrees, $Z$ can only be expressed as follows:
  \begin{align*}
  Z =  &(a_1x_1^{10} + a_2 x_2^8 + a_3 x_1^5x_2^4 + a_4 x_2y_1y_3 +a_5 x_1 y_2y_3) x_2^{12} \\
  &+ (b_1x_1^{14} + b_2 x_1^4 x_2^8 +b_3 x_1^9 x_2^4) \gamma
  \end{align*}
  By comparing both sides of the equation, it is clear that $a_1 = a_2= a_3= 0.$ Therefore
  \begin{align*}
  x_1^{16}y_2y_3 &+ 5 x_1^{15} x_2 y_1 y_3  - 5 x_1^{14}x_2^2y_1y_2 = \\
 & b_1x_1^{16}y_2y_3 -b_1 x_1^{15}x_2y_1y_3 + b_1 x_1^{14}x_2^{10}y_1y_2 \\
 &+ b_2x_1^{6}x_2^8y_2y_3 -b_2 x_1^{5}x_2^9y_1y_3 + b_2 x_1^{4}x_2^{10}y_1y_2\\
 &\; \; \; + b_3x_1^{11}x_2^4y_2y_3 -b_3 x_1^{10}x_2^5y_1y_3 + b_3 x_1^{9}x_2^{6}y_1y_2.\\
  \end{align*}
  From this equation, we get that on the one hand  $b_1=1$ and, on the other hand $b_1=-5$, which is a contradiction.
  \end{proof}

\section{The Costoya-Viruel Example}\label{sec:second}

In \cite{CV2}, the authors construct, for any given finite group $\Gamma$,
a simply-connected elliptic inflexible manifold whose group of self homotopy equivalences
is $\G$. To that end, they start with the inflexible dga  from Definition \ref{def:dga_AL}
and add generators corresponding to a certain graph $G=(V,E)$, with set of vertices $V$ and
set of edges $E$. These new generators interrelate in such a way to ensure that the
self homotopy equivalences of the dga are the automorphisms of the graph, which happens to be $\Gamma$.

To analyse the dga of \cite{CV2}, we need the following result on fibrations.

\begin{lemma}\label{lem:spec_seq}
Given a commutative diagram of simply-connected dgas
\begin{equation}\label{eq:spec_seg_dga}
\xymatrix{
(\Lambda V_1, d_1)\ar@{^{(}->}[r] \ar[d]^{\psi_1} & (\Lambda V_2, d_2) \ar@{->>}[r] \ar[d]^{\psi_2}  & (\Lambda V_3, d_3)\ar@{=}[d]^{\psi_3}\\
(\Lambda \widetilde{V}_1, \tilde{d}_1)\ar@{^{(}->}[r] & (\Lambda \widetilde{V}_2, \widetilde{d}_2)\ar@{->>}[r]   & (\Lambda V_3, d_3)
}
\end{equation}
where each $(\Lambda V_i, d_i)$, $i=1,2,3$, has cohomology with Poincar\'e duality of formal dimension $N_i$ and fundamental class $\eta_i$. If $V_2=V_1\oplus V_3$, and $\widetilde{V}_2=\widetilde{V}_1\oplus V_3$, then $H(\psi_1)(\eta_1)\ne 0$ if and only if $H(\psi_2)(\eta_2)\ne 0$.
\end{lemma}

\begin{proof}
By \cite[Section 15(a)]{FHT}, since $V_2=V_1\oplus V_3$, and $\widetilde{V}_2=\widetilde{V}_1\oplus V_3$, the
diagram in \eqref{eq:spec_seg_dga} is a Sullivan model for a commutative diagram of simply-connected rational spaces
 \begin{equation}\label{eq:spec_seg_spaces}
 \xymatrix{
X_1 & X_2 \ar[l]   & X_3\ar@{=}[d]^{\Psi_3}\ar[l]\\
\widetilde{X}_1\ar[u]^{\Psi_1} & \widetilde{X}_2\ar[l]\ar[u]^{\Psi_2}   & X_3\ar[l]
 }
 \end{equation}
where the rows are rational fibrations.
We proceed by  analysing the rational cohomology Serre spectral sequence (Sss) associated to each fibration in the diagram \eqref{eq:spec_seg_spaces}, and compare them via the induced maps connecting both rows.

As our spaces are simply-connected, the Sss associated to the top row is
 $$
 E_2^{p,q}=H^p\big(X_1;H^q(X_3;\mathbb{Q})\big)=H^p(X_1;\mathbb{Q})\otimes H^q(X_3;\mathbb{Q})\Longrightarrow H^{p+q}(X_2;\mathbb{Q}).
 $$
Since the rational cohomology of $X_i$ is concentrated in degrees at most $N_i$, $i=1,3$, then the group of highest total degree in the $E_2$-term is $E_2^{N_1,N_3}=\mathbb{Q}\{\eta_1\otimes\eta_3\}$,
and the class $\eta_1\otimes\eta_3$ survives in the $E_\infty$-term. Therefore, $N_2=N_1+N_3$ and $\eta_1\otimes\eta_3$ represents a nontrivial multiple of $\eta_2$ in the associated graded vector space given by the Sss.

We now consider $\widetilde{E}_2^{p,q}$, the Sss associated to the bottom row in diagram \eqref{eq:spec_seg_spaces}.
By means of the edge morphisms, we also know $H(\psi_1)(\eta_1)\otimes H(\psi_3)(\eta_3)=H(\psi_1)(\eta_1)\otimes\eta_3$ represents a nontrivial multiple of $H(\psi_2)(\eta_2)$ in the associated graded vector space given by the $\widetilde{E}_\infty$-term. Thus if $H(\psi_2)(\eta_2)\ne 0$,
then $H(\psi_1)(\eta_1)\otimes\eta_3\ne 0$ and $H(\psi_1)(\eta_1)\ne 0$.

Assume now $H(\psi_1)(\eta_1)\ne 0$, but $H(\psi_2)(\eta_2)= 0$.
Then $H(\psi_1)(\eta_1)\otimes\eta_3$ represents the zero class in the $\widetilde{E}_\infty$-term. Since $\widetilde{E}_n^{p,q}=0$ for $q>N_3$, $H(\psi_1)(\eta_1)\otimes\eta_3$ can only be trivial in  the $\widetilde{E}_\infty$-term if there exists $n\in\mathbb{N}$ such that $\widetilde{d}_n\big(H(\psi_1)(\eta_1)\otimes\eta_3\big)\ne 0$ in $\widetilde{E}_n^{N_1-n-1,N_3-n}$.
However,
\begin{align*}
\widetilde{d}_n\big(H(\psi_1)(\eta_1)\otimes\eta_3\big)
& = (H(\psi_1)\otimes H(\psi_3))\big(d_n(\eta_1\otimes\eta_3)\big)  && \text{ (by naturality of edge morphisms)}\\
& = (H(\psi_1)\otimes H(\psi_3))(0) && \text{ (since $0\ne\eta_1\otimes\eta_3\in E^{N_1,N_3}_\infty$)}\\
& =0,
\end{align*}
which is a contradiction. Therefore if $H(\psi_1)(\eta_1)\ne 0$, then $H(\psi_2)(\eta_2)\ne 0$.
\end{proof}

The dga of \cite{CV2} is defined as follows.

\begin{definition}\label{def:acta_example}
For a finite  and simple connected graph $G = (V, E)$ with more than one vertex,
we define the minimal Sullivan algebra associated to $G$ as
 $$
 \mathcal{M}_G= \Big(\Lambda(x_1,x_2,y_1,y_2,y_3,z)\otimes\Lambda(x_v,z_v \!\mid\! v\in V),d\Big),
 $$
 where degrees and differential are described by
\begin{alignat*}{2}
&| x_1|= 8, \qquad \qquad&&d x_1 =0,\\
&| x_2|= 10,& &d x_2 =0,\\
&| y_1|= 33,& &d y_1 =x_1^3x_2,\\
&| y_2|= 35,& &d y_2 =x_1^2x_2^2,\\
&| y_3|= 37,& &d y_3 =x_1x_2^3,\\
&| x_v|= 40,& &d x_v =0,\\
&| z|= 119, & &d z =y_1y_2x_1^4x_2^2-y_1y_3x_1^5x_2+y_2y_3x_1^6+x_1^{15}+x_2^{12},\\
&| z_v|= 119, & &d z_v =x_v^3+\sum_{(v,w)\in E}x_vx_wx_2^4.
\end{alignat*}
\end{definition}

\begin{theorem} \label{thm:acb}
Let $G = (V, E)$ be a finite  and simply-connected graph with more than one vertex. Then the minimal model $\mathcal{M}_G$ from Definition \ref{def:acta_example} is an elliptic inflexible dga of formal dimension
$N = 208 + 80 \vert V \vert$ and fundamental class $\nu=x_1^{26}\prod_{v\in V}x_v^2$.
\end{theorem}

\begin{proof}
This follows from in \cite{CV2} and \cite[Proposition I.6]{CL}.
\end{proof}

\begin{theorem}\label{thm:acta_mapsto_positive_weight}
Let $G = (V, E)$ be a finite  and simply-connected graph with more than one vertex, and let $\mathcal{M}_G$ be the minimal model from Definition \ref{def:acta_example}.
Then there exists a dga morphism $\psi\colon \mathcal{M}_G\to \bar{\mathcal{B}}$ with
$\bar{\mathcal{B}}$ a finite type dga with positive weight, and moreover  $H(\psi)([\nu]) \ne 0$ where $\nu$ is the fundamental class given
in Theorem \ref{thm:acb}. Therefore
$\mathcal{M}_G$ is not strongly inflexible. Furthermore,  any manifold for which $\cM_G$ is a Sullivan model  is not strongly inflexible either.
\end{theorem}

\begin{proof}
Let ${\cM}_{[2]}$ be the $2$-step algebra
 $$
 {\cM}_{[2]}=\Big(\Lambda(x_1,x_2,y_1,y_2,y_3)\otimes\Lambda(x_v,z_v \!\mid\! v\in V),d\Big)\subset \mathcal{M}_G,
 $$
and define $P_0=x_1^{15}+x_2^{12}$, $P_1=y_1y_2x_1^4x_2^2-y_1y_3x_1^5x_2+y_2y_3x_1^6$.
The algebra $\cM_{[2]}$ has a positive weight $\omega$ given by Proposition \ref{prop:55}, and following the arguments of
Section \ref{sec:3-step}, there exists a dga  $\cB_{\cM_G}\buildrel\text{\scriptsize def}\over{:=}(\dot \cM_{[2]} \otimes \Lambda (u_1,u_2,u_3), d)$, where
 $$
 du_1=\omega(P_0+\dot P_1),\, du_2=\omega(P_1),\, du_3=\omega(\dot P_0).
  $$
such that $\omega$ extends to a positive weight on $\cB_{\cM_G}$ given by
 $$
  \omega(u_1)=\omega(P_0+\dot P_1),\, \omega(u_2)=\omega(P_1), \, \omega(u_3)=\omega(\dot P_0).
  $$
Let $\mathcal{I}\subset\cB_{\cM_G}$ be the differential ideal generated by the $\omega$-homogeneous elements $\dot{x}_v,\dot{z}_v$. Then, according to Lemma \ref{lem:ideal_homogeneo}, the dga $\bar{\cB}=\cB_{\cM_G}/\mathcal{I}$ admits a positive weight $\widetilde{\omega}$. Define the dga morphism
 \begin{equation}\label{eqn:psi_acta}
 \begin{array}{rcl}
 \psi\colon \cM_G & \longrightarrow & \bar{\cB} ,\\
  m\in \cM_{[2]} & \mapsto & \bar{m}+\bar{\dot m},\\
  z & \mapsto & \bar{u}_1+\bar{u}_2+\bar{u}_3,
  \end{array}
 \end{equation}
and consider the commutative diagram
\begin{equation}\label{eq:spec_seg_dga_acta}
\begin{gathered}
\xymatrix{
(\cM, d) =(\Lambda(x_1,x_2,y_1,y_2,y_3,z), d\big)\ar@{^{(}->}[r] \ar[d]^{\psi_1} & \cM_G \ar@{->>}[r] \ar[d]^{\psi}  & \big(\Lambda(x_v,z_v\!\mid\! v\in V), d_3\big)\ar@{=}[d]^{\psi_3}\\
(\cB_\cM, d)\ar@{^{(}->}[r] & \bar{\cB}\ar@{->>}[r]   & \big(\Lambda(x_v,z_v\!\mid\! v\in V), d_3\big),
}
\end{gathered}
\end{equation}
where $(\cM, d)$ is the dga from Definition \ref{def:dga_AL}, and $\psi_1\colon (\cM, d)\to (\cB_\cM, d)$ is the morphism constructed in \eqref{eqn:psi}.
As we have shown in the proof of  Theorem  \ref{thm:dga_lowdegrees}, the dga $(\cM, d)$  satisfies the hypotheses of Theorem \ref{thm:cor:main1}, hence $H(\psi_1)$ maps nontrivially the fundamental class of $H(\cM)$. Notice that  $d_3(x_v)=0$, and $d_3(z_v)=x_v^3$, thus $\big(\Lambda(x_v,z_v\!\mid\! v\in V), d_3\big)$ is an elliptic dga, hence with a Poincar{\'e} duality cohomology.  Then $H(\psi)(\nu)\ne 0$ by Lemma \ref{lem:spec_seq}, and $\mathcal{M}_G$ is not strongly inflexible by Corollary \ref{cor:44}.
\end{proof}

\begin{remark}\label{rem:otherexamples}
The same arguments and calculations used to prove that the dgas from Definition \ref{def:acta_example}, \cite[Def. 2.1]{CV2}, are not strongly inflexible, work to show that the dgas   from \cite[Definition 2.1]{CosMenVir18} and
\cite[Definition 4.1]{CosMenVir20} are also not strongly inflexible. Hence, as a corollary of Theorem \ref{thm:weigth_implies_integral_flexible}, we get that any simply-connected manifold admitting one of those dgas as Sullivan minimal model is not a strongly inflexible manifold.
\end{remark}

\section{Connected sums and strong inflexibility}\label{sec:connected}

In this section we deal with more examples of inflexible manifolds that are produced from the manifolds of Sections \ref{sec:first} and
\ref{sec:second}. Using connected sums and products, infinitely many oriented closed simply-connected inflexible manifolds are constructed in \cite{Am} and \cite{CL}. We develop here all the tools to prove that building upon not strongly inflexible manifolds produces not strongly inflexible manifolds.

\begin{proposition} \label{prop:acabando}
Suppose that $M_1$ is an oriented closed manifold which is not strongly inflexible, and let $M_2$ be any other
 oriented  closed manifold. Then $M_1\x M_2$ is not strongly inflexible.
\end{proposition}

\begin{proof}
As $M_1$ is not strongly inflexible, there exists an oriented closed manifold $M'$ such that $\deg(M',M_1)$ is unbounded.
Lef $\lambda>0$ and $f\colon M'\to M_1$ such that $|\!\deg(f) | \geq \lambda$. Then $f\x \id\colon M'\x M_2 \to M_1\x M_2$ has
$\deg(f\x \id)=\deg(f)$, hence $|\!\deg(f\x \id) | \geq \lambda$. So $\deg(M'\x M_2,M_1\x M_2)$ is unbounded, and
hence $M_1\x M_2$ is not strongly inflexible.
\end{proof}

\begin{corollary}
The manifolds of \cite[Theorem II.5]{CL} are inflexible but not strongly inflexible.
\end{corollary}

\begin{proof}
 They are of the form $\prod\limits_{1\leq l\leq k} M_{l}$,
 where $M_{l}$ are taken from the examples given in Section \ref{sec:first}, which are not strongly inflexible by our previous results.  Using  Proposition \ref{prop:acabando}, we conclude.
\end{proof}

\begin{proposition}
If $M$ is a simply-connected oriented closed manifold which is not strongly inflexible, then $\#_k M$ is not strongly inflexible for all $k\geq 2$.
\end{proposition}

\begin{proof}
 Let $M'$ be a simply-connected oriented closed manifold such that $\deg(M',M)$ is unbounded.
 We apply \cite[Lemma 3.8.(2)]{SWWZ} inductively to $M,M'$ and $\#_kM,\#_kM'$ so that
  $$
  \deg(M',M) \cap \deg(\#_kM',\#_kM) \subseteq  \deg(\#_{k+1}M', \#_{k+1} M).
  $$
  Hence $\deg(M',M)\subseteq \deg(\#_kM',\#_kM)$ for all $k\geq 2$. This proves the result.
\end{proof}

\begin{corollary}
The manifolds of \cite[Example II.4]{CL} 
are inflexible but not strongly inflexible.
\end{corollary}

Finally, in \cite[Proposition II.13]{CL}, the authors construct nonspinable simply-connected inflexible manifolds.
To deal with them, we need the following:

\begin{proposition}\label{prop:555}
Let $M_1,M_2$ be oriented closed simply-connected and not strongly inflexible $N$-manifolds. Then, $M_1\# M_2$ is not strongly inflexible if there exist:

\begin{enumerate}
\item A  closed oriented simply-connected $N$-manifold $M'_2$, with $\deg(M'_2,M_2)=\ZZ$; or

\item  Oriented closed simply-connected $N$-manifolds  $M_1',M_2'$ such that
for some $s>0$,  $s\ZZ\subset \deg(M'_2,M_2)$ and $\deg(M'_1,M_1) \cap s\ZZ$ is infinite.\end{enumerate}
\end{proposition}

\begin{proof}
Let  us prove the first assertion. Let $M'_1$ be an oriented closed $N$-manifold such that $\deg(M'_1,M_1)$ is unbounded. The result follows
from \cite[Lemma 3.8.(2)]{SWWZ} as $\deg(M'_2,M_2) =\ZZ$,
implies  $\deg(M_1',M_1) \subset  \deg(M_1'\# M_2, M_1\# M_2)$. Thus $\deg(M_1'\# M_2, M_1\# M_2)$ is
unbounded and $M_1\# M_2$ is not strongly inflexible.

The second assertion is proved in a similar way.
\end{proof}

Let $s>0$ be an integer. We have the following improvement of Corollary \ref{cor:44} and Theorem
\ref{thm:weigth_implies_integral_flexible}

\begin{proposition} \label{prop:cor:44}
Let $(\cA,d)$ be a dga with Poincar\'e duality cohomology algebra and fundamental form $[\nu]\in H^N(\cA,d)$, $N\geq 4$. Suppose
that there exists a dga $(\cA',d')$ with a positive weight, and a morphism $\varphi\colon(\cA,d)\to (\cA',d')$ such that
$H(\varphi)([\nu])\neq 0$. Then there exists a  dga $(\bar\cA,d)$ whose cohomology is Poincar\'e duality of formal dimension $N$, and $\deg(\cA,\bar\cA) \cap s\ZZ$ is unbounded.

Moreover, if  $(M,\eta)$ is a simply-connected $N$-manifold with model $(\cA,d)$ as above, such that $\eta \otimes_{\mathbb Q} 1=[\nu]$, and $(\cA',d')$ is simply-connected and of finite type,
 then there exists a manifold $M'$ such that $\deg(M',M)\cap s\ZZ$ is unbounded.
\end{proposition}

\begin{proof}
Following the arguments and notation in proof of Corollary \ref{cor:44},  there exist a  dga $(\bar\cA,d)$ whose cohomology is Poincar\'e duality of formal dimension $N$, fixed numbers $0\ne\omega(a'_i)\in\N$, and $\alpha_i\in\Q$, for $i=0,\ldots,r$,
such that given any natural number $q_n$ there is a dga morphism $G_n\colon (\cA,d)\to (\bar\cA,d)$ with $$\deg(G_n)=q_n^{\omega(a'_0)}+\sum \limits_{i=1}^r q_n^{\omega(a'_i)}\alpha_i\, .$$
If $t\in\ZZ$ verifies $t\alpha_i\in\ZZ$, for every $i=1,\ldots, r$, then $q_n=nst$, $n\in\N$ makes $\deg(G_n)\in s\ZZ$. Therefore $\deg(\cA,\bar\cA) \cap s\ZZ$ is unbounded.

The last assertion follows from Theorem \ref{thm:weigth_implies_integral_flexible} that
hinges in Theorem \ref{thm:universal_implies_flexible}. There, it is shown that there exist a manifold $M'$, and integers $k\ne 0$ and $\omega(v)>0$, such that for every integer $l_n$ there exists a map $G_n\colon M'\to M$ such that $|\deg(G_n)|=|kc_nl_n^{\omega(v)}|$ for some non zero integer $c_n$. Therefore $$\{\deg(G_n) \, |\, l_n\in s\ZZ\}\subset\deg(M',M)\cap s\ZZ$$ is unbounded.
\end{proof}

All manifolds in Sections \ref{sec:first} and \ref{sec:second} satisfy the conditions of Proposition \ref{prop:cor:44},
hence the degrees of maps $\deg(M',M)\cap s\ZZ$ are unbounded for any $s>1$.

\begin{corollary}
The manifolds of \cite[Prop.\ II.13]{CL} are inflexible but not strongly inflexible.
\end{corollary}

\begin{proof}
 The manifolds of  \cite[Prop.\ II.13]{CL} are of the form $M^{\times k} \times W$, where $M$ is one of
 the manifolds in \cite[Examples I.2--I.4]{CL} (see Example \ref{ex:list_dga}), and $W=S^{N-2} \tilde{\times} S^2$ is the total space of the sphere bundle of the nontrivial
rank $(n-1)$–vector bundle over $S^2$.

 With $s=2$, there is some $M'$ such that $\deg(M',M^{\times k})\cap 2\ZZ$ is
 unbounded. Now taking the pull-back under the degree $2$ map $g:S^2\to S^2$ of the nontrivial bundle
 $S^{N-2}\to W \to S^2$ we get a trivial bundle $\tilde W=g^*W \cong S^{N-2}\x S^2$. As this
 has self-maps of any degree, we have that $\deg(\tilde W, W) \supset 2\ZZ$.
 Applying now Proposition \ref{prop:555}, we get that $M^{\times k}\times W$ is not strongly inflexible.
\end{proof}

For completeness, we will include some results on degrees of maps for connected sums, using dgas.
First, given subsets $A,B\subset \ZZ$, define
 $$
 A+B\buildrel\text{\scriptsize def}\over{:=}\{a+b \, | \, a\in A, b\in B\}\subset\ZZ.
 $$

\begin{lemma}\label{lem:grado_de_suma_conexa}
Let $M_i$, $i=1,2,3$, be (not necessarily simply-connected) closed oriented $N$-manifolds.
Then
$$\deg(M_1,M_3)+\deg(M_2,M_3)\subset \deg(M_1\# M_2,M_3).$$
\end{lemma}

\begin{proof}
Let $q\colon M_1\# M_2\to M_1\vee M_2$ denote the pinching map. Then for any given maps $f_i\colon M_i\to M_3$, the map $f\colon M_1\# M_2\to M_3$ given by the composition $f=  (f_1\vee f_2)\circ q$ verifies $\deg(f)=\deg(f_1)+\deg(f_2)$, and the result follows.
\end{proof}

Note that the inclusion in Lemma \ref{lem:grado_de_suma_conexa} can be strict.  Consider, for example,  $M_1=M_2 =T_2$, the
$2$-torus, and   $M_3 = T_2 \#  T_2$. Then $\deg (M_i,  M_3) = 0$ for $i=1,2$, whereas $1 \in \deg (M_3, M_3).$
\begin{corollary}\label{cor:infinitas_variedades_para_flexibles}
Let $M$ be a closed oriented and not strongly inflexible $N$-manifold.
Then there exist infinitely many closed oriented $N$-manifolds $M'$ such that $\deg(M',M)$ is unbounded.
\end{corollary}

\begin{proof}
Since $M$ is not strongly inflexible, there exists a closed oriented $N$-manifold $M'$ such that $\deg(M',M)$ is unbounded.
Then, according to Lemma \ref{lem:grado_de_suma_conexa} for any $N$-manifolds $W$,
$\deg(M',M)\subset \deg(W\# M',M)$, and therefore $\deg(W\# M',M)$ is unbounded too.
\end{proof}

The following is a straightforward consequence of the previous corollary.

\begin{corollary}\label{cor:seminormas_infinitas}
Let $M$ be a closed oriented $N$-manifold, and let $v_M$ be the domination  $\operatorname{Mfd}_N$-seminorm associated with $M$
(see \cite[Definition 7.1]{CL}). If $v_M$ is not finite, then there exist infinitely many  closed oriented $N$-manifolds $M'$ such that $v_M(M')=\infty$.
\end{corollary}

\begin{proof}
Recall that given an oriented closed $N$-manifold $M'$,
 $$
 v_M(M')= \sup\{|d|\, \, |\, d\in\deg(M',M)\},
 $$
thus the result follows from Corollary \ref{cor:infinitas_variedades_para_flexibles}.
\end{proof}

\begin{definition}\label{def:connected_sum_dga}
Let $\mathcal{A}_i$, $i=1,2$, be connected dgas, and let $a_i\in\mathcal{A}_i$ be elements such that $|a_1|=|a_2|$. The connected sum of the pairs $(\mathcal{A}_i,a_i)$, $i=1,2$, is the dga
 $$
 (\mathcal{A}_1,a_1)\#(\mathcal{A}_2,a_2) \buildrel\text{\scriptsize def}\over{:=} (\mathcal{A}_1\oplus_{\mathbb{Q}}\mathcal{A}_2 )/I\, ,
 $$
where $\mathcal{A}_1\oplus_{\mathbb{Q}}\mathcal{A}_2\buildrel\text{\scriptsize def}\over{:=}(\mathcal{A}_1 \oplus\mathcal{A}_2)/\mathbb{Q}\{(1,-1)\}$, and
$I\subset \mathcal{A}_1\oplus_{\mathbb{Q}}\mathcal{A}_2$ is the differential ideal generated by $a_1-a_2$.
\end{definition}

The connected sum of dgas provides a model for the connected sum of manifolds.

\begin{theorem}\label{thm:model_connected_sum}
Let $M_i$, $i=1,2$, be simply-connected closed oriented $N$-manifolds. Let $(\mathcal{M}_i, m_i)$ be the associated rational model
with fundamental class. Then $(\mathcal{M}_1, m_1)\# (\mathcal{M}_2, m_2)$ is a rational model of $M_1\# M_2$.
\end{theorem}

\begin{proof}
If $(\mathcal{M}_i, m_i)$ is a rational model of $M_i$, then $\mathcal{M}_1\oplus_{\mathbb{Q}}\mathcal{M}_2$ is a rational model of $M_1\vee M_2$ \cite[Example 2.47]{FOT}, and therefore dividing by the ideal generated by $m_1-m_2$ is a rational model of $M_1\# M_2$ \cite[Example 3.6]{FOT}.
\end{proof}

The connected sum of dgas behaves nicely with positive weights.

\begin{proposition}\label{prop:connected_vs_pesos}
Let $\mathcal{A}_i$, $i=1,2$, be connected dgas, and let $a_i\in\mathcal{A}_i$ be elements such that $|a_1|=|a_2|$. If both $\mathcal{A}_1$ and $\mathcal{A}_2$ have positive weights, namely $\omega_1$ and $\omega_2$, such that $a_i$ is $\omega_i$-homogeneous, $i=1,2$,  then $(\mathcal{A}_1,a_1)\#(\mathcal{A}_2,a_2)$ admits a positive weight.
\end{proposition}

\begin{proof}
Let $\omega$ be the positive weight in $\mathcal{A}_1\oplus_{\mathbb{Q}}\mathcal{A}_2$ given by
$$
\omega(x)=
\begin{cases}
\omega_2(a_2)\omega_1(x), &\text{if $x\in \mathcal{A}_1$ is $\omega_1$-homogeneous,}\\
\omega_1(a_1)\omega_2(x), &\text{if $x\in \mathcal{A}_2$ is $\omega_2$-homogeneous.}
\end{cases}$$
Then $a_1-a_2\in\mathcal{A}_1\oplus_{\mathbb{Q}}\mathcal{A}_2$ is $\omega$-homogeneous, hence $d(a_1-a_2)$ is so. Therefore since
  $$
  I=(\mathcal{A}_1\oplus_{\mathbb{Q}}\mathcal{A}_2) \cdot (a_1-a_2, d(a_1-a_2)),
  $$
the differential closed ideal $I$ is generated (as vector space) by $\omega$-homogeneous elements, and according to Lemma \ref{lem:ideal_homogeneo}, $\omega$ gives rise to a positive weight on $(\mathcal{A}_1,a_1)\#(\mathcal{A}_2,a_2)$.
\end{proof}

Finally,

\begin{theorem}\label{thm:connected_sum_via_pesos}
Let $(M_i,\eta_i)$ be a simply-connected $N$-manifold with minimal model $\mathcal{M}_i=(\Lambda V_i,d)$, $i=1,2$. Write the cohomological fundamental class as $\eta_i=[\nu_i]$.
Assume there exist  dga morphisms $\psi_i\colon \mathcal{M}_i\to \mathcal{A}_i$, $i=1,2$, such  that $\mathcal{A}_i$ is a finite type dga that has a positive weight and $H(\psi_i)(\eta_i) \ne 0$. Then $M_1\# M_2$ is not strongly inflexible. 
\end{theorem}

\begin{proof}
Let $\omega_i$ be a positive weight on $\mathcal{A}_i$, $i=1,2$. We claim $\psi_i(\nu_i)$ may be assumed $\omega_i$-homogeneous.
As in the proof of Theorem \ref{thm:universal_implies_flexible}, decompose $\psi_i(\nu_i)=\sum \limits_{ j=0}^r \alpha_j$ into
$\omega_i$-homogeneous elements, fix $\widetilde{a}_i\in \mathcal{A}_i$ with nontrivial $\alpha_j$, $\omega_i$-homogeneous
complements $\mathcal{A}_i=\la \widetilde{a}_i\ra \oplus W_i$, and define
 $$
 I_i=  \mathcal{A}_i^{\geq N+1} \oplus W_i\, .
 $$
Then
by Lemma \ref{lem:ideal_homogeneo}, $\widetilde{\mathcal{A}}_i=\mathcal{A}_i/I_i$ is a finite type connected dga with positive weight,
and the induced morphism $\widetilde{\psi}_i\colon \mathcal{M}_i\to \widetilde{\mathcal{A}}_i$ maps
$\nu_i$ to a $\widetilde{\omega}_i$-homogeneous element.

The algebra $\big(\widetilde{\mathcal{A}}_1,\widetilde{\psi}_1(\nu_1)\big)\#\big(\widetilde{\mathcal{A}}_2,
\widetilde{\psi}_2(\nu_2)\big)$ inherits a positive weight by Proposition \ref{prop:connected_vs_pesos}. Therefore the obvious morphism
 $$
 \widetilde{\psi}_1\# \widetilde{\psi}_2\colon (\mathcal{M}_1,\nu_1)\#(\mathcal{M}_2,\nu_2)\to
 \big(\widetilde{\mathcal{A}}_1,\widetilde{\psi}_1(\nu_1)\big)\#\big(\widetilde{\mathcal{A}}_2, \widetilde{\psi}_2(\nu_2)\big)
$$
is a morphism from the rational model of $(M_1\# M_2, \eta)$ to a finite type dga that has a positive weight,
such that $[(\widetilde{\psi}_1\# \widetilde{\psi}_2)(\nu) ] \ne 0$, where $\eta=[\nu]$ is the fundamental class.
Finally, according to Theorem \ref{thm:weigth_implies_integral_flexible},  $(M_1\# M_2, \eta)$ is not strongly inflexible.
\end{proof}

\end{document}